\documentclass[leqno,  a4paper, twoside]{conm-p-l}

\usepackage[cmtip,all]{xy}

\usepackage[colorlinks,linkcolor=blue,pagebackref=true,citecolor=red]{hyperref}

\usepackage{mathptmx}
\usepackage{newtxtext}

\usepackage{amsmath}
\usepackage{amssymb}
\usepackage{amsthm}

\usepackage{soul}
\usepackage[mathscr]{eucal}
\usepackage{dutchcal}
\usepackage{graphicx}

\usepackage{dsfont}			
\usepackage{extarrows}

\usepackage{pv-macro-arxiv}

 \usepackage[flushmargin]{footmisc}

\begin{document}
	

\thispagestyle{empty}

\title[Rational points and trace forms]{Rational points and trace forms on a finite algebra over
a real closed~field$\,{}^{\dag}$}

\author[Patil-Verma]{
Dilip\,P.\,Patil$^{\,1}$\,${}^{\ast}$}
\address{$^{1}$\,Department of Mathematics, Indian Institute of Science Bangalore}
\email{patil@iisc.ac.in}
\thanks{$^{\bf \dag}$\,{This expository article  on  the Pederson-Roy-Szpirglas theorem  about counting real roots of real polynomial equations.   Most of the exposition is influenced by the discussions of the first author with late Prof.\,Dr.\,Uwe Storch (1940-2017) and the lecture courses delivered by him.}}
\thanks{$^{\bf \ast}$\,{During the preparation of this work the  first author was visiting Department of Mathematics, Indian Institute of Technology Bombay. He would like to express his gratitude for the generous financial  support from IIT Bombay and encouraging cooperation. He was also partially supported by  his project MATRICS-DSTO-1983 during the final preparation of this manuscript.}}

\author[]{J.\,K.\,Verma$^{\,2}$}
\address{Indian Institute of Technology Bombay, Mumbai, INDIA 400076}
\email{jkv@math.iitb.ac.in}

\date{\today}

\footnotesep=3mm

\cameraready



\maketitle


\begin{abstract}
The main goal of this article is to provide a proof of the Pederson-Roy-Szpirglas theorem  about counting common real zeros of  real polynomial equations by using basic results from Linear algebra and Commutative algebra.
The main tools are symmetric bilinear forms, Hermitian forms, trace forms and their invariants such as rank, types and signatures. Further, we use
the equality (proved in \cite{BS}) of the number of $K$-rational points of a zero-dimensional  affine algebraic set  over a real closed field $K$ with the signature of the trace form of its coordinate ring to prove the Pederson-Roy-Szpirglas theorem, see\,\cite{PRS1993}.
\end{abstract}

\section{Introduction}

The objective of this paper is to present an exposition of classical and modern results con- cerning  the number of real or complex points in the solution space of a finite system of poly\-nomial equations with real coefficients in arbitrary number of variables. More precisely, for polynomials $F_{1}, \ldots , F_{m} \in\R[X_{1},\ldots , X_{n}]$, assume that the residue-class $\R$-algebra $\R[X_{1},\ldots , X_{n}]/\langle F_{1},\ldots , F_{m}\rangle$ is finite dimensional over $\R$, then the set of common~zeros \\[1mm]
\hspace*{7mm} $\,{\rm V}_{\R}(F_{1},\ldots , F_{m}):=\{(a_{1}, \ldots , a_{n})\in\R^{n}\mid F_{j}(a_{1}, \ldots , a_{n})=0 \ \hbox{ for all } \ j=1,\ldots , m\}\,$ \\[1mm]
%
of $F_{1},\ldots , F_{m}$ in $\R^{n}\!$ is finite. The converse is not true,  for example, for $F_{1}\!=\!X_{1}^{2}\!+\!1$, ${\rm V}_{\R}(F_{1})\!=\!\emptyset$ is finite and $\R[X_{1},\ldots , X_{n}]/\langle F_{1}\rangle\!\!\iso\!\!\C[X_{2},\ldots , X_{n}]$ is not finite dimensional over $\R$ if $n\!\geq\!2$. However,
for polynomials $F_{1}, \ldots , F_{m}\!\in\!\C[X_{1},\ldots , X_{n}]$,  the residue-class $\C$-algebra $\C[X_{1},\ldots , X_{n}]/\langle F_{1},\ldots , F_{m}\rangle$ is finite dimensional over $\C$ if and only if the set of common zeros
${\rm V}_{\C}(F_{1},\ldots , F_{m})$   of $F_{1}, \ldots , F_{m}$ in $\C^{n}$ is finite.
Moreover, by  the classical Hilbert's nullstellensatz  ${\rm V}_{\C}(F_{1},\ldots , F_{m})\neq \emptyset$ if and only if the ideal $\langle F_{1},\ldots , F_{m}\rangle$ generated by $F_{1},\ldots , F_{m}$ in $\C[X_{1},\ldots , X_{n}]$ is a non-unit ideal. But, this is not true over the field $\R$ or more generally over real closed fields. Therefore
the natural questions one  deals with are when exactly ${\rm V}_{K}(F_{1},\ldots , F_{m})\neq \emptyset$  and how to find its cardinality, where $K$ is an arbitrary real closed field. \\[1mm]
Many researchers have studied these  problems and devised effective algorithms.  For example, already in the 19th century Sturm, Jacobi,  Sylvester, Hermite,  Hurwitz proved fundamental results for  counting real points (in small number of variables $n\!\leq\!2$) by using the signature of appropriate quadratic forms.
\smallskip

In Section\,2 and Section\,3, we collect standard results on symmetric bilinear  and Hermitian forms over a  real closed field $K$ and its algebraic closure $\C_{K}\!=\!K[\,{\rm i}\,]$ with ${\rm i}^{2}\!=\!-1$. However, for the sake of completeness, we recall them without proofs in the format they are used in later sections. With these preliminaries at the end od Section\,3, we state the important Rigidity theorem for quadratic forms (see \cite{BS}) which is used in Section\,4.
\smallskip

In Section\,4, we collect some elementary concepts from commutative algebra and
recall the important Theorem\,4.5 from \cite{BS} which relates the $K$-rational points of a finite dimensional algebra $A$ over a real closed field $K$ with the type of the trace form $\tr_{K}^{A}$ on $A$ and derive some consequences.

\smallskip

In Section\,5, we compute the cardinality of the $K$-rational points of finite algebra over real closed field $K$. The main ingredient in this section is the Shape Lemma\,5.3 which guarantees a distinguished generating set for a $0$-dimensional  radical ideal $\mathfrak{A}\subseteq K[X_{1},\ldots , X_{n}]$ from which  one can reduce the problem of counting the number of $K$-rational points in ${\rm V}_{K}(\mathfrak{A})$ to the one variable case.
In Theorem\,\ref{thm:5.5} using the results from Section\,4,  we relate type, signature and rank of a generalized trace forms on $A=K[X_{1}, \ldots , X_{n}]/\mathfrak{A}$ with the number of points in ${\rm V}_{K}(\mathfrak{A})$ and in
${\rm V}_{\overline{K}}(\mathfrak{A})$. Finally, we give a proof of theorem of Pederson-Roy-Szpirglas\,\cite[Theorem\,2.1]{PRS1993}.

\section{Decomposition theorem for   Hermitian forms}

The main aim of this section is to recall the Decomposition Theorem (see\,\ref{decthm:2.13}) which guarantees  the existence of orthogonal bases (with respect to Hermitian forms). For this we recall basic concepts and steps which lead to its proof.  Most of these results can be found in standard graduate text books, for instance see \cite[Ch.\,V, \S12]{StorchWiebe2011},  \cite[Ch.\,IX]{SchejaStorch1988}  or \cite[Ch.\,11]{Jacobson1989}, \cite[Ch.\,7]{Artin1994}, or \cite[Ch.\,XV]{Lang2002}. However, for setting the notation, terminology and for the sake of completeness, we recall them without proofs  in the format that they are used in this as well as in the later sections.
\smallskip


\begin{notass}\label{notass:2.1}
In order to define   symmetric and Hermitian forms toge\-ther and prove results about them,  we fix the following convenient notation\,:
\smallskip

Let $K$ be a field and let  $\kappa:K\rightarrow K$ be a fixed \textit{involution} (an auto\-mor\-phism whose square is the identity, i\,.e. its inverse is  itself). We denote by $K':=K^{\kappa}:=\{a\in K\mid \kappa(a) =a\}\subseteq K$ the fixed field of $K$. There are exactly two cases\,: (i)\, $\kappa =\id_{K}$ and (ii)\, $\kappa \neq \id_{K}$. In this first case, we assume that ${\rm Char}\,K\neq 2$. \\[0.75mm]
The   involution $\kappa$ of $K$ is simply denoted by the standard bar-notation $\kappa:K\rightarrow K$, $a\mapsto \overline{a}$ and called the  \so{conjugation of} $K$.
 Therefore we have\,: $\,\overline{a+b} =\overline{a}+ \overline{b}$, $\,\overline{ab} =\overline{a}\overline{b}$ and $\overline{\overline{a}} =a$  for all $a,b\in K$.  Furthermore,  the fixed field $K'= K$ in the first case and $K'\subsetneq K$ in the second case.
\end{notass}
\smallskip

\begin{examples}\label{exs:2.2}
\begin{arlist}
\item
For an  arbitrary field $K$, the identity map $\id_{K}:K\to K$ is an involution. For $K\!=\!\R$,  the identity  $\id_{\R}\!$ is the only  involution. For $K\!=\!\C$, besides the identity  $\id_{\C}$, the usual complex conjugation $ \C\rightarrow \C$, $z\mapsto \overline{z}$, is the only other involution of  $\C$ which play an important role\,\footnote{\label{foot:1}\,In the case,  $V=W=\C$, the distance of a point $z$ from the origin is not given by  the bilinear form $(z,w)\mapsto z\cdot w$, but by using the map  $(z,w)\mapsto z\cdot \overline{w}\,$, namely,  $|z|=\sqrt{z\,\overline{z}}$.}.

\item
The complex-conjugation is a special case of the conjugation of a quadratic algebra $A$ over an arbitrary field $K$\,: If $1, \omega\in A$ is a $K$-basis of $A$ with $\omega^{2}=\alpha +\beta \omega$, $\alpha$, $\beta\in K$, then the conjugation $A\rightarrow A$ of $A$ is defined by $\overline{a+b\,\omega} = (a+b\beta)-b\,\omega$, $a$, $b\in K$. It is easy to see this is  an involution of the $K$-algebra $A$ and is not equal to $\id_{A}$.
For an arbitrary element $x\in A$, the norm, the trace and the characteristic polynomial of $x$ over $K$ are defined by the  equations\,:\, ${\rm N}^{A}_{K}(x)= x\overline{x}$, $\,\Tr_{K}^{A}(x) = x+\overline{x}$, $\,\chi_{x}= X^{2}-(x+\overline{x})\,X+x\overline{x} = (X-x)(X-\overline{x})$, respectively.
\\[1mm]
There are many examples of this type, for example,
if $L$ is a field and if $\kappa \in \textrm{Aut}\, L$ is an involution of $L$ with $\kappa\neq \id_{L}$, then $L$ is a quadratic algebra over the fixed field $K:=L^{\kappa}:=\{a\in L\mid \kappa(a)=a\}$  and the involution $\kappa$ of $L$ coincides with the conjugation of the quadratic algebra over $K$ and the Galois group $\,\Gal(L\,\vert\,L^{\kappa}) =\textrm{Aut}_{L^{\kappa}\hbox{\scriptsize -alg}}\, L =\{\id_{L}, \kappa\}$. A typical example of this type is the algebraic closure $\C_{K}=K[\,{\rm i}\,]$, where ${\rm i}^{2}=-1$, of a real closed field $K$.
\end{arlist}

\end{examples}
\smallskip

\begin{definitions}\label{def:2.3}
Let $V$ and $W$ be $K$-vector spaces.

\begin{arlist}
\smallskip

\item
A map $f:V\rightarrow W$ of $K$-vector spaces is called \so{semilinear} (or \so{conjugate-linear}) (with respect to the conjugation of $K$) if $f$ is additive and $f(ax) = \overline{a}x\,$ for all $a\in K$ and all $x\in V$. The semilinear maps from $V$ into $W$ coincide with the $K$-linear maps from $V$ into the \so{anti-vector space} $\overline{W}$ corresponding to $W$ and also with the $K$-linear maps from the $K$-vector space $\overline{V}$ into $W$, where $\overline{W}$ (resp. $\overline{V}$) is a $K$-vector spaces with the scalar multiplication $(a,y)\mapsto \overline{a}y$ defined by using the given scalar multiplication on $W$ (resp. $V$).
\smallskip

 \item
 A function  $\Phi\colon V\times W\to K$ is called \so{sesquilinear}  if $\Phi$ is  $K$-linear in the first component and semilinear (with respect to the conjugation  of $K$) in the second component, i.\,e. if for all  $a$, $a'\in K$ and all  $x$, $x'\in V$,  $y$, $y'\in W$, we have\,:

\begin{alist}
\item $\,\Phi (ax+a'x',y)=a\,\Phi (x,y)+a'\,\Phi (x',y)$.
\item  $\,\Phi (x,ay+a'y')=\overline{a}\,\Phi (x,y)+\overline{a'}\,\Phi (x,y')$.
\end{alist}
\smallskip

\noindent The set of sesquilinear functions $V\times W \to K$ is denoted by ${\rm Sesq}_{K}(V,W)$ which is clearly a subspace of the $K$-vector space $K^{V\times W}\!$. If $V=W$, then  sesquilinear functions are also called \so{sesquilinear forms} on $V$.  Note that if the conjugation of $K$ is equal to  $\id_{K}$, then the sesquilinear functions are linear in both variables, i.\,e. they are bilinear and hence  $\Sesq_{\,K}(V,W) =\Mult_{K}(V,W)$ (\,= the set of bilinear functions).
\smallskip

\noindent The bijective map $\Sesq_{\,K}(V,W) \longrightarrow (V\otimes_{K}\overline{W})^\ast$, $\,\Phi\longmapsto \left(x\otimes y \mapsto \Phi(x,y)\right)$ is an isomorphism of $K$-vector spaces (with inverse $\,\varphi\longmapsto \left((x,y)\mapsto \varphi(x\otimes y) \right)$.

\end{arlist}

\end{definitions}
\smallskip

\begin{mypar}\label{mypar:2.4}{\sc Gram's Matrix}\,
Let $V$,   $W$ be finite dimensional $K$-vector spaces with  bases $\mbcalx:=\{x_i \mid i\in I\}$, $I$ finite indexed  set,   $\mbcaly:=\{y_j\mid j\in J\}$, $J$  finite indexed set, respectively.
Then every sesquilinear  function $\Phi\colon V\times W\rightarrow K$ is uniquely determined by the values $\Phi(x_{i}, y_{j})$, $(i,j)\in I\times J$. Conversely, for arbitrary family $c_{ij}\in K$, $(i,j)\in I\times J$, there is a (unique) sesquilinear   function   $\Phi \colon V\times W\to K$ defined by $\Phi(\sum_{i\in I} a_{i}x_{i}, \sum_{j\in J} b_{j}y_{j}) := \sum_{(i,j)\in I\times J} a_{i}\, {b}_{j}\, c_{ij}$.  Moreover, the map
\vspace*{-1mm}
\begin{align*}
\hspace*{15mm} {\rm Sesq}_{K}(V,W) \longrightarrow K^{I\times J} ={\rm M}_{I,J}(K)\,\,,  \ \Phi\longmapsto \left(\Phi(x_{i}, y_{j})\right)_{(i,j)\in I\times J}\,,
\vspace*{-2mm}
\end{align*}
is an isomorphism of $K$-vector spaces.
\smallskip \\
For a sesquilinear  function $\Phi:V\times W \longrightarrow K$ the $I\times J$-matrix \\[1mm]
\hspace*{30mm} $\displaystyle \,\mathscr{G}_{\Phi}(\mbcalx\,;\, \mbcaly):= \left(\Phi(x_{i}, y_{j})\right)_{(i,j)\in I\times J}\in {\rm M}_{I,J}(K)\,$\\[1mm] is called the \so{Gram's matrix}
or the \so{fundamental matrix}
of $\Phi$ with respect to the bases $\mbcalx$ and $\mbcaly$. If $V=W$ and $\mbcalx=\mbcaly$, then we simply write $\,\mathscr{G}_{\Phi}(\mbcalx)$. Further, if $I=J$, then the determinant ${\rm G}_{\Phi}(\mbcalx, \mbcaly):=\Det\,\mathscr{G}_{\Phi}(\mbcalx\,;\, \mbcaly)$ is called the \so{Gram's determinant} with respect to the bases $\mbcalx$ and $\mbcaly$.  If $V=W$ and if $y_{j}=x_{j}$ for all $j\in J=I$,  we simply write ${\rm G}_{\Phi}(\mbcalx)$.
\smallskip

For the computation with Gram's matrices, it is convenient to extend the conjugation of $K$ to matrices
over $K$\,:\,
For a matrix $\mscrA=\left(a_{ij}\right)\!\in\!{\rm M}_{I,J}(K)$ with $I$,  $J$ finite indexed sets,  put $\,\overline{\mscrA}\!:=\!\left(\overline{a}_{ij}\right)\!\in\!{\rm M}_{I, J}(K)$. Then
the map ${\rm M}_{I, J}(K)\rightarrow\!{\rm M}_{I, J}(K)$,
$\mscrA\mapsto \overline{\mscrA}$, is a semilinear (with respect to the conjugation of $K$)  involution of the $K$-vector space ${\rm M}_{I,J}(K)$.
Further, $^{\rm t}\overline{\mscrA} = \overline{^{{\rm t}}{\mscrA}}$
 (where for a $I\times J$-matrix $\mscrA\in\M_{I,J}(K)$, ${}^{\rm t}\mscrA$ denote the transpose of $\mscrA$)
and $\overline{\mscrA\mscrB} =\overline{\mscrA}\,\overline{\mscrB}$ if $\mscrB\in {\rm M}_{J\, R}(K)$, $R$ finite indexed set.
For a square matrix $\mscrA\in {\rm M}_{I}(K)$, $\Det \overline{\mscrA} = \overline{\Det \mscrA}$ and if  $\,\mscrA\in \GL_{I}(K)$, then  $\overline{\mscrA}^{-1}\! = \!\overline{\mscrA^{-1}}$.
\smallskip \\
{\ref{mypar:2.4}.1}\,
{\it Let  $\mbcalx'=(x'_{i})_{i\in I}$,  $\mbcaly'=(y'_{j})_{j\in J}$ be another $K$-bases of $V$,  $W$ and  $\mscrA=(a_{ri})\in \GL_{I}(K)$,  $\mscrB=(b_{sj})\in \GL_{J}(K)$ be the transition matrices of the bases from $\mbcalx$ to $\mbcalx'$,  from $\mbcaly$ to $\mbcaly'$, respectively. Then for a sesquilinear  function $\Phi:V\times W \rightarrow K$, we have  the \so{\rm transformation formula}\,$:$\,\\[1mm]
\hspace*{17mm} $\,{\mscrG}_{\Phi}({\mbcalx}\,;\,{\mbcaly}) =
{}^{{\rm t}}{\mscrA}\,{\mscrG}_{\Phi}({\mbcalx}'\,;\,{\mbcaly}')\,\overline{\mscrB}\quad\hbox{or}\quad
{\mscrG}_{\Phi}({\mbcalx}'\,;\,{\mbcaly}') =
{}^{{\rm t}}{\mscrA}^{-1}\,{\mscrG}_{\Phi}({\mbcalx}\,;\,{\mbcaly})\,\overline{\mscrB}^{-1}\!\!.$ \\[1mm]
In~particular, if  $V=W$, then
$\,{\mscrG}_{\Phi}(\mbcalx)\! = {}^{\rm t}{\mscrA}\,{\mscrG}_{\Phi}({\mbcalx}')\,\overline{\mscrA}$.}

\end{mypar}
\smallskip

\begin{examples}\label{exs:2.5}
\begin{arlist}
\item
\so{(Standard forms})\,   The \so{standard form} on the standard $K$-vector space $K^{(I)}$, $I$ an indexed set,  with the standard basis $e_{i}, i\in I$  of $K^{(I)}$  is the
sesquiliner  form defined by the unit matrix $\mscrE_{I}\in {\rm M}_{I}(K)$
and is denoted by $\langle -,-\rangle$, that is,
$\langle e_{i}, e_{j}\rangle =\delta_{ij}$ for all $i,j\in I$.
Therefore
$\,\bigr\langle (a_i)\,,(b_i)\,\bigr\rangle = \sum_{i\in I} a_i\,\overline{b}_i = {}^{\rm t}{\mscrA}\,\mscrB\,$, where $\,{\mscrA}:=(a_i)\,, {\mscrB}: = (b_i)\in K^{(I)}$ (are column vectors).
In~particular, if $I = \{1,\ldots ,n\}$, then $\langle (a_{1},\ldots , a_{n}), (b_{1},\ldots , b_{n})\rangle  = a_1 \overline{b}_1 + \cdots + a_n  \overline{b}_n$.
\smallskip

\item \so{(Natural Duality)}\,
Let $K$ be an arbitrary field with $\id_{K}$ as conjugation,  $V$ a $K$-vector space and let
$V^*={\rm Hom}_{K}\,(V,K)$ denote the dual space of $V$.
The canonical evaluation map $\,\mathscr{E}:V\times V^*\longrightarrow K$, $\,(x,f)\longmapsto \langle x,f\rangle :=  f(x)$,
is a bilinear and is called the
\so{natural duality} between $V$ and $V^*$. If $V$ is finite dimensional with basis
$\mbcalx=\{x_1, \ldots , x_{n}\}$, and if ${\mbcalx}^{*}=\{x_1^*, \ldots , x^*_{n}\}$ is  the corresponding dual basis, then the Gram's  matrix of this natural duality $\mscrG_{\mathscr{E}}(\mbcalx\,,\,{\mbcalx}^{*}) =\mscrE_{n}$ is the unit matrix in $\M_{n}(K)$.
\end{arlist}
\end{examples}
\smallskip

\begin{mypar}\label{mypar:2.6}{\sc Non-degeneracy  and Complete Duality}\,
An important motivation for the study of sesquilinear functions is  {\it the description of linear form through vectors.} (See Example\,\ref{ex:2.9}).
\smallskip

Let $V$ and $W$ be  $K$-vector spaces and let  $\Phi\colon V\times W\longrightarrow K$ be a  sesquilinear
function. The canonical semilinear maps defined by\,: \,\\[0.5mm]
\hspace*{3mm} $\,\Phi_1 \colon V \longrightarrow W^*\, \  x\longmapsto (\,y\mapsto \overline{\Phi(x,y)}\,)\quad{\rm and}\quad
\Phi_2\colon W \longrightarrow V^*\, \  y\longmapsto \left(\,x\mapsto \Phi(x,y)\,\right)$, \\[0.5mm]
are simple denoted by $\Phi_{1}(x) = \Phi(x,-)$ and $\Phi_{2}(y)= \Phi(-,y)$.

Further, from each one of the map $\Phi_{1}$ resp.  $\Phi_{2}$, one can recover $\Phi$, since
$\,\Phi (x,y) = \overline{\left(\Phi_1 (x)\right)(y)}  = \left(\Phi_2 (y)\right)(y)\,$ for all $x\in V$ and all $y\in W$.
\smallskip

\ref{mypar:2.6}.1\, {\it Suppose that  both $V$, $W$ are finite dimensional over $K$ with bases    ${\mbcalx}=(x_i)_{i\in I}$,   ${\mbcaly}=(y_j)_{j\in J}$, respectively. Then the matrices of the canonical semilinear maps
$\Phi _1$ and $\Phi_2$  with respect to bases $\mbcalx$, ${\mbcaly}^{*}$ and $\mbcaly$, ${\mbcalx}^{*}$, where  ${\mbcalx}^*$ and ${\mbcaly}^*$ are dual bases of $\mbcalx$ and $\mbcaly$, respectively, are  given by\,$:$\\[1mm]
\hspace*{30mm} $\,{\mscrM}^{{\mbcalx}}_{{\mbcaly}^{*}}(\Phi _1) = {}^{\rm t}\overline{{\mscrG}_{\Phi}({\mbcalx}\,;\,{\mbcaly})}\quad\hbox{and}\quad
{\mscrM}^{{\mbcaly}}_{{\mbcalx}^{*}}(\Phi_2) = {\mscrG}_{\Phi}({\mbcalx}\,;\,{\mbcaly})$.} \\[1mm]
Further, since taking the transpose and conjugation,  the rank of a matrix is unaltered, both $\Phi_{1}$ and $\Phi_{2}$ have the same rank. This common rank  of the maps  $\Phi_1$, $\Phi_2$ is called the  \so{rank} of the sesquilinear function  $\Phi$ and is denoted by  ${\rm rank}\,\Phi$. Therefore, {\it $\,{\rm rank}\,\Phi$ is the rank of the Gram's  matrix of $\Phi$ with respect to arbitrary bases of $V$ and $W$.}

\end{mypar}
\smallskip

The case when  $\Phi_1$ and $\Phi_2$ are both injective or both bijective are important\,:

\begin{definition}\label{def:2.7}
Let  $V$ and $W$ be $K$-vector spaces and let $\Phi\colon V\times W\longrightarrow K$ be a  sesquilinear function. We say that
\smallskip \\
(1) $\,\Phi$ is   \so{non-degenerate} if  $\Phi_1$ and $\Phi_2$ are both injective.\\[0.5mm]
(2)  $\,\Phi$ defines a  \so{complete duality} (between  $V$ and $W$) if  $\Phi_1$ and $\Phi_2$ are both bijective.
\end{definition}
\smallskip

\begin{example}\label{ex:2.8}\so{(Trace form)}\,
Let $V$ be a finite dimensional $K$-vector space. The map $\End_{K} V \times \End_{K} V$,  $\,(f,g)\longmapsto \tr (fg),$ is a  symmetric bilinear form on the $K$-vector space ${\rm End}_KV$ of $K$-endomorphisms of $V$ and is called the \so{trace form} on ${\rm End}_KV$.
\smallskip

\ref{ex:2.8}\,1\,
{\it Let $V$ be a finite dimensional $K$-vector space. Then  the trace form defines a complete duality on ${\rm End}_KV$.}
\smallskip

Let  $A$ be a finite (dimensional) $K$-algebra. For an element $x\in A$, $\lambda_{x}:A\to A$ denote  the left multiplication map  on $A$ by $x$. Then the map $A\times A \to K$,
$\,(x,y)\mapsto \tr_K^A(xy)= \tr (\lambda_x\lambda_y)$,
defines  a symmetric bilinear form on  $A$ and is called the \so{trace form} of the $K$-algebra $A$.
The trace form on $A$ reflects many important properties of the $K$-algebra  $A$, see Section\,4.
\smallskip

Note that if $A={\rm End}_{K} V$ , then  the trace form on the $K$-algebra  ${\rm End}_{K} V$ is different from the above introduced trace form on the $K$-vector space   ${\rm End}_{K} V$. Obviously, for every endomorphism $f\in{\rm End}_{K} V$\,:
$\,\tr_K^{{\rm End}_{K} V}f = n\,\cdot \tr f\,$, $\enskip n:={\rm Dim}_KV\,$.
\end{example}
\smallskip

\begin{example}\label{ex:2.9}(\,\so{Gradient of a linear form}\,)\,
For a finite dimensional $K$-vector space $V$ the natural duality between $V$ and
$V^*$ (see Example\,\ref{exs:2.5}\,(2)) is a complete duality, since
its Gram's matrix with respect to dual bases is the unit matrix. Further,
in this case $\Phi_{1}:V\rightarrow (V^{*})^{*}$ is the canonical evaluation map $x\mapsto \mathscr{E}_{x}: f\mapsto \mathscr{E}_{x}(f):=f(x)$ and $\Phi_{2}:V^{*}\rightarrow V^{*}$ is the identity map $\id_{V^{*}}$.
 In~particular,  for  every linear form   $\varphi\colon V^*\to K$,  there exists a unique vector $x\in V$ (which is independent on $\varphi$) such that $\varphi (f)=f(x)$ for every $f\in V^{*}$, i.\,e.  $\varphi$ is the evaluation of the linear forms in  $V^*$ at the vector $x\in V$. \\[1mm]
If $V$ is not finite dimensional, then the natural duality is non-degenerate but never complete.
This follows from the fact that  for every  $x\in V$ (also in the infinite dimensional case) can be extended to a basis $V$ and hence  there exists a linear form $f\in V^*$ with $\langle x,f \rangle = f(x)\ne 0$.
\smallskip

More generally, if $\,\Phi:V\times W \longrightarrow K$ defines a complete duality, then one can use $\Phi_{1}$ and $\Phi_{2}$ to identify the $K$-vector spaces $V$ and $W^{*}\!\!$, resp. $W$ and $V^{*}\!\!$. Therefore, for every linear form $f\in W^{*}\!\!$, there exists a unique vector $x_{f}\in V$ with $f = \overline{\Phi(x_{f}, -)}$, and for every linear form $\varphi\in V^{*}\!\!$, there exists a unique vector $y_{\varphi}\in W$ with $\varphi=\Phi(-, y_{\varphi})$. The vectors $x_{f}$ resp. $y_{\varphi}$ are called the \so{gradients of}  $f$ resp. $\varphi$ (with respect to $\Phi$) and are denoted by ${\rm grad}\,f$ resp. ${\rm grad}\,\varphi$.  Therefore the linear forms on $W$ resp. $V$ correspond to their respective gradients, i.\,e. $f=\Phi_{1}({\rm grad}\,f)$ and
$\varphi=\Phi_{2}({\rm grad}\,\varphi)$.
\end{example}
\smallskip

 \begin{mypar}\label{mypar:2.10}{\sc Orthogonality,    perpendicular  relation  and  Hermitian forms}\, The concept of orthogonality has its origin in Euclidean geometry.
 \smallskip

 Let  $V$ and $W$ be $K$-vector spaces and let $\Phi\colon V\times W\longrightarrow K$ be a  sesquilinear function. \smallskip \\
(1)\, The vectors $x\in V$ and $y\in W$ are called \so{orthogonal} or \so{perpendicular} to each other with respect to $\Phi$ if $\Phi(x,y)=0$. In this case we write $x\bot_{\Phi}\, y$ or simply $x\bot y$ (if $\Phi$ is fixed).\\[0.75mm]
(2) Two subsets  $M\subseteq V$ and $N\subseteq W$ are called  \so{orthogonal}
if  $x\bot y$ for all  $x\in M$ and for all $y\in N$. In this case we write  $M\bot N$. Futher, we put \\[1mm]
 \hspace*{25mm} $M^\bot:=\{y\in W\mid M\bot\{y\}\}\quad\hbox{and}\quad {}^\bot N:=\{x\in V\mid \{x\}\bot N\}$. \\[1mm]
Obviously, $M^\bot$ and ${}^\bot\, N$ are $K$-subspaces of $W$ and $V$, resp.
\smallskip

For example, if  $f\in W^{*}$ is a linear form with a gradient (see Example\,\ref{ex:2.8})
${\rm grad}\,f\in V$, i.\,e. $f(y) = \overline{\Phi({\rm grad}\,f, y)}$  for all $y\in W$, then
$\Ker f = \{{\rm grad}\,f\}^\bot$. Analogously, if  a linear form $\varphi\in V^{*}$ with a gradient ${\rm grad}\,\varphi \in W$, then $\Ker \varphi= {}^\bot\,\{{\rm grad}\,\varphi\}$.
\medskip

Note that if  $\Phi:V\times V \longrightarrow K$ is a sesquilinear form on $V$, then
for a subset $M\subseteq V$, the subsets $M^\bot \!=\!\{y\in V\mid M \bot \{y\}\}$ and ${}^\bot\,M\!=\!\{x\in V\mid  \{x\}\bot M\}$ are {\it not} equal in general, since the relation of perpendicularity is not symmetric.
To remove this difference, one considers the  symmetric   (resp.  Hermitian,  skew-Hermitian)  forms if the conjugation of $K$ (see  \ref{notass:2.1}) is $\id_{K}$ (resp. $\neq \id_{K}$)\,:
\smallskip

 \ref{mypar:2.10}.1\, {\sc Definition}\,   Let $V$ be a finite dimensional  $K$-vector space and let $\Phi:V\times V \longrightarrow K$ be a sesquilinear form on $V$,  We say that $\Phi$ is \so{Hermitian} (resp. \so{skew-Hermitian}) if $\Phi(x,y) =\overline{\Phi(y,x)}$  (resp.
 $\Phi(x,y) = -\overline{\Phi(y,x)}$) for all $x$, $y\in V$. \\[1mm]
 With the notation and assumptions in \ref{notass:2.1}, we note that  the term ``Hermitian '' and ``skew-Hermitian''  mean ``symmetric'' and ``skew-symmetric'' if the conjugation of $K$ is $\id_{K}$.  If the conjugation of $K$ is $\neq \id_{K}$, then sometimes we use the terms ``pure-Hermitian'' and ``pure-skew-Hermitian'' forms on $V$. In the case $K=\C$ with the usual complex-conjugation, the Hermitian (resp. skew-Hermitian) forme are simply called
 \so{complex-Hermitian} (resp. \so{complex-skew-Hermitian}).
 \smallskip

 \ref{mypar:2.10}.2\,\,  {\it If $\,\Phi_1 \colon V \longrightarrow V^*$ and $\Phi_2\colon V \longrightarrow V^*$ are  the canonical semilinear maps  associated  to the sesquilinear form $\Phi$ on $V$, {\rm see \ref{mypar:2.6}}, then   $\Phi$ is Hermitian $($resp. skew-Hermitian$)$ if and only if $\Phi_{1}=\Phi_{2} ($resp. $\Phi_{1}=-\Phi_{2})$.}\\[1mm]
 Further, for a  Hermitian (resp. skew-Hermitian) form $\Phi$ on $V$, the relation $\bot$ on $V$ is symmetric. In this case the subspace ${}^\bot V=V^\bot=\Ker \Phi_{1}=\Ker \Phi_{2}$  of $V$ is called the \so{degeneration space}\, or, \, the \so{radical} of $\Phi$ and is also denoted by ${\rm Rad}\,(V,\Phi)={\rm Rad}\,\Phi$.
 \smallskip

\ref{mypar:2.10}.3\,\,
{\it A sesquilinear form $\Phi:V\times V \to K$  on a finite dimensional $K$-vector space $V$  is Hermitian $($resp. skew-Hermitian$)$ if  and only if  the Gram's matrix $\mscrG_{\Phi}(\mbcalx) = (\Phi (x_i,x_j)) \in {\rm M}_{I}(K)$ of $\,\Phi$ with respect to every basis $\mbcalx=\{x_{i}\mid i\in I\}$  of $V$  is Hermitian $($resp. skew-Hermitian$)$.} \\[1mm]
Recall that a square matrix $\mscrA\in \M_{I}(K)$, $I$ finite indexed set,  is \so{Hermitian} (resp. \so{skew-Hermitian} if $\,\mscrA ={}^{\rm t}\overline{\mscrA}$ (resp. $\mscrA =-{}^{\rm t}\overline{\mscrA}$.
A matrix $\mscrA\in \M_{I}(K)$ is \so{symmetric} (resp. \so{skew-symmetric} if $\,\mscrA ={}^{\rm t}{\mscrA}$ (resp. $\mscrA =-{}^{\rm t}{\mscrA})$.
 \smallskip

 \ref{mypar:2.10}.4\,\,   {\it Let $V$ be a finite dimensional $K$-vector space
 and let ${\mbcalx}\!=\!\{x_{i}\mid i\in I\}$ be a basis of  $V$. {\rm (recall that $K'$ is the fixed field of $K$, see\,\ref{notass:2.1})}
 The map
 $\,\Phi\longmapsto {\mscrG}_{\Phi}(\mbcalx)\,$ is a $K'$-linear isomorphism of the $K'$-vector space of Hermitian
 $($resp. skew-Hermitian$)$ forms on $V$ onto the
$K'$-vector space of Hermitian $($resp. skew-Hermitian$)$  matrices in ${\rm M}_{I}(K)$.
Moreover, if
${\mbcalx}'=\{x'_{i}\mid i\in I\}$ is another basis of $V$ with transition matrix $\mscrA=(a_{ij})\in \GL_{I}(K)$, {\rm  i.\,e.}$x_{j}= \sum_{i\in I}\,a_{ij}x'_{i}$, then the Gram's matrices ${\mscrG}_{\Phi}({\mbcalx})$ and ${\mscrG}_{\Phi}({\mbcalx}')$ are related by the rule\,: \\[0.5mm]
\hspace*{20mm} ${\mscrG}_{\Phi}({\mbcalx}) = {}^{\rm t}\mscrA\,{\mscrG}_{\Phi}({\mbcalx}')\,\overline{\mscrA}\quad$   or
$\quad {\mscrG_{\Phi}({\mbcalx}'}) = {}^{\rm t}\mscrA^{-1}\,{\mscrG}_{\Phi}({\mbcalx})\,\overline{\mscrA}^{-1}$.}
\medskip

\ref{mypar:2.10}.5\,\, In important cases a sesquilinear  forms  on a $K$-vector space $V$ are  completely determined by its values on the  \so{diagonal} $\,\Delta_{\,V} = \{(x,x) \bigm| x\in V\}$. More precisely, we have\,:\, \\[1mm]
\ref{mypar:2.10}.5a\,\,  {\sc Polarisation identity}\,
Let   $V$ be a $K$-vector space. Then\,:\, \\[1mm]
(1) If  the conjugation (see\,\ref{notass:2.1}) of $K$ is $\neq \id_{K}$, then for every sesquilinear form $\Phi:V\times V \to K$, for all $x,y\in V$  and $a\in K$ with $\overline{a}\neq a$,  (using Cramer's rule) we have\,:\\[0.5mm]
\hspace*{2mm} $\Phi (x,y)\!=\!\textstyle{\frac{1}{a-\overline{a}}}\,\bigl(\Phi (ax+y\,, ax+y)\!-\!\overline{a}\,\Phi (x+y,x+y)\!-\!\overline{a}(a\!-\!1)\,\Phi (x,x)\!-\!(1\!-\!\overline{a})\,\Phi (y,y)  \bigr)$\,and \\[0.75mm]
\hspace*{2mm} $\Phi (y,x)\!=\!\textstyle{\frac{1}{\overline{a}-a}}\,\bigl(\Phi (ax+y\,, ax+y)\!-\!a\,\Phi (x+y,x+y)\!-\!a\,(\overline{a}\!-\!1)\,\Phi (x,x)\! -\!(1\!-\!a)\,\Phi (y,y) \bigr)\,$. \\[1mm]
(2) If ${\rm Char}\,K\neq 2$, then for  every symmetric bilinear $\Phi:V\times V \rightarrow K$ form  and for all $x,y\in V\!$, we have\,: \,\\[0.5mm]
\hspace*{2mm} $\Phi (x,y)\!=\!\textstyle{\frac{1}{2}}\,\bigl(\Phi (x+y\,, x+y)\!-\!\Phi (x, x)\!-\!\Phi (y, y) \bigr)\!=\!\textstyle{\frac{1}{4}}\,\bigl(\Phi (x+y\,, x+y)\!-\!\Phi (x-y, x-y) \bigr)$.
\medskip

\ref{mypar:2.10}.5b\,\,  {\sc Corollary}\,
{\it If  the  conjugation of $K$ is $\neq \id_{K}$ {\rm (see \ref{notass:2.1})} and $V$ is a $K$-vector space, then a sesquilinear form  $\Phi:V\times V \to K$ on a $K$-vector space $V$ is  Hermitian $($resp. skew-Hermitian$)$  if and only if $\Phi(x,x) = \overline{\Phi(x,x)}\,($resp. $\Phi(x,x) = -\overline{\Phi(x,x)})$
for all $x\in V$, {\rm i.\,e.}  $\Phi(x,x)\in K'$
{\rm (the fixed field of $K$ with respect to the conjugation of $K$, see\,\ref{notass:2.1})}.  In~particular, a complex-sesquilinear form is complex-Hermitian $($resp. complex-skew-Hermitian$)$ if and only if the values   the values $\Phi(x,x)$, $x\in V$,  are all real $($resp.  purely-imaginary$)$.}
\medskip

\ref{mypar:2.10}.5c\,\, {\sc Corollary}\,
{\it  Let $V$ be vector space over the field $K$ of characteristic $\neq 2$.  Then a symmetric bilinear form
 $\Phi\colon V\times V\to K$ on $V$ is the  zero form if and only if $\,\Phi(x,x)=0\,\,$  for all $x\in V$.}

\end{mypar}
\smallskip

\begin{mypar}\label{mypar:2.11}\,{\sc Orthogonal  direct  sums}\,
Let $V_{i}$, $i\in I$; $W_{i}$, $i\in I$ be two families of $K$-vector spaces and let $\,\Phi_{i}:V_{i}\times W_{i}\to K$ be a family of sesquilinear functions. Then the map
\vspace*{-1mm}
\begin{align*}
\Phi:\bigl(\bigoplus_{i\in I} V_{i}\bigr)\times \bigl(\bigoplus_{i\in I} W_{i}\bigr) \longrightarrow K\,, \, \left((x_{i}), (y_{i})\right)\longmapsto \sum_{i\in I} \Phi_{i}(x_{i}, y_{i})
\vspace*{-2mm}
\end{align*}
is a sesquiinear function and its restrictions $\,\Phi\,\vert\,V_{i}\times W_{i} = \Phi_{i}$ for all $i\in I$, where $V_{i}$ (resp. $W_{i}$) is considered canonically as subspace of $\bigoplus_{i\in I} V_{i}$ (resp. $\bigoplus_{i\in I} W_{i}$). Further, $V_{i}\bot W_{j}$ with respect to $\Phi$ for all $i,j\in I$, $i\neq j$. This sesquilinear function $\Phi$ is called the  \so{orthogonal direct sum} of the family $\Phi_{i}$, $i\in I$ and is denoted by  $\,\displaystyle \operp_{i\in I}\Phi_{i}$.
Conversely,  if $\Phi:V\times W\to K$ is a sesquilinear function and $V$(resp. $W$) is a direct sum of the $K$-subspaces $V_{i}$, $i\in I$ (resp. $W_{i}$, $i\in I$) with $V_{i}\perp W_{j}$ for all $i$, $j\in I$, $i\neq j$ with respect to $\Phi$, such that $\Phi(\sum_{i\in I} v_{i}, \sum_{j\in I} w_{j})= \sum_{i\in I} \Phi(v_{i}, w_{i})$ for $v_{i}\in V_{i}$, $w_{j}\in W_{j}$. Then we say that $\Phi$ is the orthogonal direct sum of the $\Phi_{i}:= \Phi\,\vert\,V_{i}\times W_{i}$, $i\in I$. In~particular, if $V=W$ and $V_{i}=W_{i}$, $i\in I$, then  $V$ is the \so{orthogonal direct sum of the subspaces} $V_{i}$, $i\in I$, with respect to $\Phi$ and is denoted by $V=\operp_{i\in I}V_{i}$.
\end{mypar}
 \smallskip

\begin{mypar}\label{mypar:2.12}{\sc Orthogonal basis}\,\,
Let $V$ be a $K$-vector space and let $\Phi:V\times V\rightarrow K$ be a sesquilinear form on $V$.
A  a family of vectors $x_i$, $i\!\in\!I$,  in $V$ is called  \so{orthogonal} with respect to $\Phi$ if $x_i\perp x_j$ for all  $i$, $j\!\in\!I$ with $i\neq j$.  Moreover, if  $\Phi(x_i,x_i)\!=\!1$ for all  $i\!\in\!I$, then it is called \so{orthonormal} with respect to $\Phi$.
\smallskip

If  $x_i$, $i\in I$,  is an orthogonal basis of $V$ with respect to $\Phi$, then
$V$ is the orthogonal  direct sum of the $1$-dimensional subspaces  $Kx_i$, $i\in I$. Moreover, if $I$ is finite, then the Gram's matrix of $\Phi$ with respect to the basis $x_i$, $i\in I$,  is the diagonal matrix ${\rm Diag}(\Phi(x_i,x_i))_{i\in I}$.
\smallskip

The following Decomposition Theorem\,\ref{decthm:2.13} guarantees the existence of orthogonal bases (with respect to Hermitian forms)  is the starting point for the classification of Hermitian forms\,:
\end{mypar}
\smallskip

\begin{decthm}\label{decthm:2.13}\,
{\it Let $K$ be a field with notations and assumptions as in {\rm \ref{notass:2.1}} and let   $\Phi:V\times V\rightarrow K$ be a sesquilinear form on a finite dimensional $K$-vector space $V$.   Then in each of  the following cases\,:  \\[1mm]
{\rm (a)}\,  The conjugation of $K$ {\rm (see\,\ref{notass:2.1})} is $\neq \id_{K}$ and $\Phi$ is Hermitian or skew-Hermitian, \\
{\rm (b)}\, ${\rm Char}\, K\neq 2$ and $\Phi$ is a symmetric bilinear form, \\[0.5mm]
$V$ has an orthogonal basis $\mbcalx\!=\!\{x_{1}, \ldots , x_{n}\}$, $n\!=\!\Dim_{K}V$, with respect to $\Phi$. In otherwords\,$:$\, $V$ is the ortho\-gonal direct sum  $V\!=\!\operp_{i=1}^{n} Kx_{i}$ into $1$-dimensional subspaces $Kx_{i}$, $i\!=\!1,\ldots , n$, with respect to $\Phi$ and $\Phi = \operp_{i=1}^{n} \Phi\vert Kx_{i}$ is the orthogonal direct sum of its restrictions $\Phi\vert Kx_{i}$, $i=1,\ldots , n$. \\[1mm]
{\rm Matrix formulation\,:}\, If either $\mscrG\in \M_{I}(K)$, $I$ finite indexed set,  is a Hermitain or skew-Hermitian matrix, or if the conjugation is $=\id_{K}$,  ${\rm Char}\,K\neq 2$ and  $\mscrG$ is a symmetric matrix, then there exists an invertible matrix $\mscrA\in\GL_{I}(K)$ such that ${}^{\rm t}\mscrA\mscrG\overline{\mscrA}$ is a diagonal matrix.
}
 \end{decthm}

 \begin{proof}
 Use  induction on $\Dim_{K} V$,  \ref{mypar:2.10}.5a, \ref{mypar:2.10}.5b and \ref{mypar:2.10}.5c.
 \end{proof}
\smallskip

\begin{mypar}\label{mypar:2.14}{\sc Automorphisms and Congruence}\,
Let $\Phi :V\times V\rightarrow K$ and   $\Psi :W\times W\rightarrow K$  be sesquilinear forms on the $K$-vector spaces $V$ and $W$, respectively. A map  $f\,\colon V\rightarrow W$ is called a  \so{homomor\-phism}
of  $(V,\Phi)$ in $(W,\Psi)$ if it is $K$-linear and is compatible with the forms $\Phi$ and $\Psi$, i.\,e.
$\Phi(x,y)=\Psi(f(x),f(y))$ for all  $x,y\in V$.  A bijective homomorphism  $f:(V,\Phi)\to (W,\Psi)$ is called an  \so{isomorphism} of $(V,\Phi)$ onto $(W,\Psi)$.
\smallskip

A  homomorphism $(V,\Phi)\to (V,\Phi)$  of is called an   \so{endomorphism} of $(V,\Phi)$ or of $\Phi$. The set of endo\-morphisms $\End_{K}(V,\Phi)$ (with composition) is a monoid.  An   isomorphism $(V,\Phi)\rightarrow(V,\Phi)$ is called an  \so{automorphism} of $(V,\Phi)$ or of $\Phi$. The set of automorphisms $\Aut_{K}(V,\Phi)$ of $(V,\Phi)$ is the unit group of the monoid  $\End_{K}(V,\Phi)$ and is called the \so{automorphism group} of  $(V,\Phi)$ or  of $\Phi$ .
\smallskip

If there exists an isomorphism from $(V,\Phi)$ onto $(W,\Psi)$, then $(V,\Phi)$ and $(W,\Psi)$ or also the forms $\Phi$ and $\Psi$ said to be \so{congruent}. If  $f: (V,\Phi)\rightarrow (W,\Psi)$ is an isomorphism, then  the map ${\rm Aut}\,\Phi \to {\rm Aut}\,\Psi$,  $g\mapsto fgf^{-1}$ is an isomorphism of groups.
\smallskip

Two square matrices  ${\mscrC}$, ${\mscrC}'\in {\rm M}_n(K)$ are said to be \so{congruent} if there exists an invertible matrix ${\mscrA}\in {\rm GL}_{n}(K)$ with  ${\mscrC} = {}^{\rm t}{\mscrA}\,{\mscrC}'\,{{\mscrA}}$.
\smallskip

On  sesquilinear forms on finite dimensional $K$-vector spaces (resp. of square matrices over $K$) the relation of  ``being congruent'' is an equivalence relation.
\medskip

\ref{mypar:2.14}.1\,
{\it Let $V$,  $W$ be finite dimensional $K$-vector spaces with  $\Dim_{K} V\!=\!\Dim_{K} W$,  $\mbcalx\!=\!\{x_i\mid i\in I\}$,   $\mbcaly\!=\!\{y_i\mid  i\in I\}$ bases of $\,V$, $W$ and let  $\Phi$,  $\Psi$  be sesquilinear forms on $V$,   $W$, resp.
Further,  let $f:V\rightarrow  W$ be an  $K$-isomorphism  of vector spaces.
Then $f$ is an isomorphism $(V,\Phi) \rightarrow (W,\Psi)$ if and only if
$\,{\mscrG}_\Phi(\mbcalx) = {}^{\rm t} {\mscrM^{\mbcalx}_{\mbcaly}(f)}\,{\mscrG}_\Psi(\mbcaly)\,{\overline{\mscrM^{\mbcalx}_{\mbcaly}(f)}}$.
In~particular,   $\Phi$ and $\Psi$ are congruent if and only if there exists
$\,{\mscrA}\!\in\!\GL_I(K)$ with  ${\mscrG}_\Phi(\mbcalx) = {}^{\rm t}{\mscrA}\,{\mscrG}_\Psi(\mbcaly)\,{\overline{\mscrA}}$.}
\smallskip

\ref{mypar:2.14}.2\, {\sc Corollary}\,
Let  $\Phi$ be a sesquilinear form on a finite dimensional $K$-vector space $V$ with basis $\mbcalx=\{x_{i}\mid i\in I\}$. Then an automorphism $f\in\Aut_{K}V$  is an automorphism of $(V,\Phi)$ if and only if $\,\mscrG_{\Phi}({\mbcalx}) =
{}^{\rm t}{\mscrM^{\mbcalx}_{\mbcalx}(f)}\,\mscrG_{\Phi}({\mbcalx})\,\overline{\mscrM^{\mbcalx}_{\mbcalx}(f)}$.
\end{mypar}
\smallskip

\begin{mypar}\label{mypar:2.15}{\sc Classification\,  problem\,  for\,  sesquilinear\, forms}\,
The classification problem for the ses\-quilinear forms on finite dimensional $K$-vector spaces is to find a
\textit{well arranged}  representative system for the equivalence classes of congruent sesquilinear forms.\\[1mm]
For example, from the Decomposition Theorem\,\ref{decthm:2.13} it follows immediately that\,:\\[1mm]
\ref{mypar:2.15}.1\,
{\it Let $K$ be a field with notation and assumptions as in {\rm \ref{notass:2.1}}. Then  every pure-Hermitian or pure-skew-Hermitian $($resp. if $\,{\rm Char}\,K\neq 2$, then every symmetric$)$ matrix is congruent to a diagonal matrix $\Diag(c_{i})_{i\in I}\in \M_{I}(K)$, $I$ finite indexed set.
Moreover,  the  form  $K^{I}\times K^{I} \rightarrow K$   defined by  $(e_i, e_j)\mapsto\delta_{ij}c_i$, where $e_{i}$, $i\in I$, is the standard basis of $K^I\!$ is also  congruent to the form defined by
$((a_i),(b_i))\mapsto\sum_{i\in I}a_i \overline{b}_i c_i$. } \\[1mm]
The above  form defined in \ref{mypar:2.15}.1 (and every other form which is congruent to this form) defined by the diagonal matrix $\Diag(c_{i})_{i\in I}\in \M_{I}(K)$ is denoted by  $[c_i]_{i\in I}$ and for $I=\{1,\ldots,n\}$ also by  $[c_1,\ldots,c_n]$.
The form  $[c_i]_{i\in I}$ is the orthogonal direct sum of the forms  $[c_i]\,:(a,b)\mapsto a\overline{b}c_i$, $i\in I$,  on $K$, therefore\,: \,\\[1mm]
\hspace*{45mm} $[c_i]_{i\in I}=\operp_{i\in I}\,[c_i]$.
\medskip

In general,  it is difficult to classify the forms  $[c_i]_{i\in I}$ up to congruence. Obviously, the form   $[c_i]_{i\in I}$  is congruent to the form $[a_i^{2}c_i]_{i\in I}$, where  $a_i\in K^\times\!\!$,  $i\in I$, since this is the transition of the basis   $e_i$, $i\in I$,  to the basis  $a_ie_i$, $i\in I$. Therefore, one can replace the elements  $c_i$ (if non-zero)  by their images in the residue class group $\,K^\times/ {\rm N}(K^\times\!)$ of $K^\times$ modulo of the subgroup
${\rm N}(K^{\times}\!):=\{a\overline{a}\mid a \in K^{\times}\}$. Note that ${\rm N}(K^{\times}\!)\subseteq K'^{\times}\!\!$, where $K'$ is the  fixed field of the conjugation of $K$ (see\,\ref{notass:2.1}) and if the conjugation is $\id_{K}$, then ${\rm N}(K^{\times}\!) ={}^{2}K^{\times}\!$ is the subgroup of the quadratic-units in $K$.

\end{mypar}
\smallskip

If the form $[c_{i}]_{i\in I} $ is Hermitian, then $c_{i}\in K'$ for every $i\in I$ and from the Decomposition Theorem\,\ref{decthm:2.13}, we have\,:
\smallskip

\begin{theorem}\label{thm:2.16}$\!\!\!$
Let $K$ be a field with notation and assumptions as in {\rm \ref{notass:2.1}} and
${\rm N}(K^{\times}\!) = K'^{\times}\!\!$,   where  $K'$ is the fixed field  of the conjugation of $K$.    Then every Hermitian form of rank $r$ on an $n$-dimensional $K$-vector space is congruent to the form  $\,[1,\ldots,1,0,\ldots,0]$,  where  $1$ occurs $r$ times and $0$ occurs  $n-r$ times. In~particular, every non-degenerate Hermitian form on an $n$-dimensional $K$-vector space is congruent to the standard form $[1,\ldots,1]$ on  $K^n$.
\end{theorem}

\begin{corollary}\label{cor:2.17}
Let $K$ be a field with ${\rm Char}\,K\neq 2$ and $K={}^{2}K$
{\rm (e.\,g., if $K$ is alg\-ebrai\-cally closed, $K=\C$)}. Then  all symmetric matrices in ${\rm M}_{n}(K)$ of equal rank are congruent. The diag\-onal matrices
$\,{\rm Diag}\,(0,\ldots,0)$, ${\rm Diag}\,(1,0,\ldots,0)$, $\ldots,$ \
${\rm Diag}\,(1,\ldots,1)=\mscrE_n\,$
form a complete representative system for the congruence classes of symmetric matrices over $K$.
\end{corollary}


\section{Type and signature of  Hermitian forms}

In this section, we recall the classification of symmetric and Hermitian forms on finite dimensional vector spaces over a real closed field and its algebraic closure, up to congruence, see\,Definition\,\ref{mypar:2.10}.1.
Most of these results can be found in the  graduate text books,  \cite[Ch.\,V, \S12]{StorchWiebe2011},  \cite[Ch.\,IX]{SchejaStorch1988}  or \cite[Ch.\,11]{Jacobson1989}, \cite[Ch.\,7]{Artin1994}, or \cite[Ch.\,XV]{Lang2002}. However,  for the sake of completeness, we recall them without proofs.


\begin{mypar}\label{mypar:3.1}{\sc Notation}\,\,(See also \ref{notass:2.1})\,
Let $K$ be  a \textit{real closed field}\,\footnote{\label{foot:2}\,{\bf Real closed fields}\, A field $K$ is called  \so{real closed} if it is \textit{real}, i.\,e.
 if for all $a_1,\dots,a_n \in K$,   \  $a_1^2\!+\!\cdots\!+\!a_n^2\!=\!0$ implies $a_1\!=\!\cdots\!=\!a_n\!=\!0$.
and if it has  no nontrivial real algebraic extension $L\,\vert\, K,$ $L \neq K$.  For example, the field $\R$ of real numbers is real closed. The algebraic closure of $\Q$ in $\R$ is real closed. The field $\Q$  is  real, but not real closed. In 1927 Artin-Schreier proved\,: {\it A field $K$ is real if and only if there is an order $\leq$ on $K$ such that $(K,\leq)$ is an ordered field.} In~particular, the characteristic of a real field is $0$.
\vskip1pt
\noindent {\bf Theorem}\, (\so{Euler-Lagrange})\,
{\it Let $(K,\leq)$ be an ordered field satisfying the properties\,:\,  {\rm (i)}  Every polynomial $f\in K[X]$ of odd degree has a zero in $K$.\, {\rm (ii)\,}  Every positive element in $K$ is a square in $K$. Then the field $\overline{K}=K({\rm i})$ obtained from $K$ by adjoining  a square root ${\rm i}$ of $-1$ is algebraically closed. In~particular, $K$ itself is real-closed.} For a proof see\,\cite[Ch.\,11, \S11.1]{Jacobson1989}.
({\bf Remark\,:} Since the field $\R$ of real numbers is ordered and satisfies the properties (i) and (ii), the  Eulae-Lagrange  theorem proves the \so{Fundamental Theorem of Algebra}\,:  {\it The field $\C=\R({\rm i})$ of complex numbers is algebraically closed .} The Euler-Lagrange Theorem  has a remarkable complement\,:\,---\,{\bf Theorem}\, (\so{Artin}\,-\,\so{Schreier})\,
{\it Let $L$ be an algebraically closed field. If $K\subseteq L$ be a subfield of $L$ such that $\,L\,\vert\,K$ is finite and $K\neq L$, then $L=K({\rm i})$ with ${\rm i}^{2}+1=0$ and $K$ is  a real-closed field.}  For a proof see\,\cite[Ch.\,11, \S11.7]{Jacobson1989}.)}. Then $\Aut K=\{\id_{K}\}$ and the field $\C_{K}:=K[{\rm i}\,]$, where ${\rm i}^{2}\!=-1$, of (complex) numbers over $K$, is the algebraic closure of $K$ with the Galois group $\Gal(\C_{K}\,\vert\,K)= \{\id_{\C_{K}}, \kappa\}$, where $\kappa :\C_{K}\to \C_{K}$, is the (complex)-conjugation defined by by ${\rm i}\mapsto -{\rm i}\,$, see\,\ref{notass:2.1} and Examples\,\ref{exs:2.2}.   \\[0.75mm]
Further, we denote by $\K$ either the field $K$,  or\,   the field $\C_{K}$. The involution in the case $\K=K$ is   $\id_{K}$ and in the case $\K=\C_{K}$ is the (complex)-conjugation  $\kappa:\C_{K}\to \C_{K}$ which we will simply denote by the standard bar-notation, i.\,e. $a\mapsto \overline{a}$, $a\in \C_{K}$.\\[0.75mm]
With these notation and assumptions, we note that  the term ``Hermitian ''  means ``real-symmetric'' in the case $\K=K$ and  ``complex-Hermitian'' if $\K= \C_{K}$.
\end{mypar}

\begin{proposition}\label{prop:3.2}
With the notation as in {\rm \ref{mypar:3.1}},
let $\Phi:V\times V\longrightarrow \K$ be a Hermitian  form on the finite dimensional $\K$-vector space $V$. Then there exists an orthogonal basis $\mbcalx=\{x_{1}, \ldots , x_{n}\}$,  $n:=\Dim_{K} V$ of $V$ with respect to $\Phi$ such that the Gram's matrix $\mscrG_{\Phi}(\mbcalx)$ is a diagonal matrix\,: \\[0.75mm]
\hspace*{30mm} $\,\mscrE_{n}^{p,q}:={\rm Diag}\,(\underbrace{1,\ldots ,1}_{p\,\hbox{\scriptsize -times}},
 \underbrace{-1,\ldots ,-1}_{q\,\hbox{\scriptsize -times}}, \underbrace{0,\ldots ,0}_{n-p-q\,\hbox{\scriptsize -times}})$.
 \end{proposition}

\noindent Below in \ref{thm:3.4} we note that that the numbers $p$ and $q$  are uniquely determined by $\Phi$. Obviously, $p+q=\rank\,\Phi$. In~particular, $\Phi$ is non-degenerate if and only if $p+q=n$. For characterization of invariants $p$ and $q$ the following concepts are useful\,:
\smallskip

\begin{definition}\label{def:3.3}
Let $\Phi:V\times V\longrightarrow \K$ be a Hermitian  form on the finite dimensional $\K$-vector space $V$. Then $\,\Phi$ is called\,:
\vskip2pt

(1)  \so{positive definite} if $\Phi (x,x)> 0$ for all  $x\in V$, $x\ne 0$.
\vskip1pt

(2)   \so{negative definite} if $\Phi (x,x)< 0$ for all  $x\in V$, $x\ne 0$.
\vskip1pt

(3)   \so{positive semi}{\,-\,}\so{definite} if $\Phi (x,x)\geq  0$ for all  $x\in V$.
\vskip1pt

(4)   \so{negative semi}{\,-\,}\so{definite} if $\Phi (x,x)\leq  0$ for all  $x\in V$.
\vskip1pt

(5)   \so{indefinite} if there are vectors $x,y\in V$ with $\Phi (x,x)> 0$ and $\Phi (y,y)< 0$.

\end{definition}

\begin{mypar}\label{thm:3.4}{\sc Sylvester's Law of Inertia}\, {\it Let $\Phi$ be a Hermitian  form on the finite dimensional $\K$-vector space $V$  and let $\mbcalx=\{x_1,\ldots ,x_n\}$, $n:=\Dim_{\K} V$ be an   orthogonal  basis of $V$ with respect to $\Phi$  such that  the Gram's matrix $\mscrG_{\Phi}(\mbcalx)$ of $\Phi$ is the diagonal matrix\,:\\[0.75mm]
\hspace*{30mm}
$\,\mscrE_{n}^{p,q}:={\rm Diag}\,(\underbrace{1,\ldots ,1}_{p\,\hbox{\scriptsize -times}},
 \underbrace{-1,\ldots ,-1}_{q\,\hbox{\scriptsize -times}}, \underbrace{0,\ldots ,0}_{n-p-q\,\hbox{\scriptsize -times}})$.\\[1mm]
Then  $p$ is the maximum of the dimensions of  subspaces  of $V$ on which $\Phi$ is positive definite, and $q$ the maximum of the dimensions of subspaces of $V$ on which $\Phi$ is negative definite. ---\,In~particular,  $p$ and $q$ do not depend on the special choice of the orthogonal basis $x_1,\ldots ,x_n$ of $\,V$.}
\end{mypar}

\begin{definition}\label{def:3.5}
The pair $(p,q)$ as in the Sylvester's Law of Inertia\,\ref{thm:3.4} is called the \so{type} of the form $\Phi$. The natural number   $p$ is called the  (\so{inertia}\,-)\,\so{index}, the integer $p-q$ is called the  \so{signature} and the natural number $q$ is called the  \so{Morse}\,-\,\so{index} of $\Phi$. We denote the rank, signature and type of a Hermitian  form $\Phi$ by $\rank\,\Phi$, $\sign\,\Phi$ and $\type\,\Phi$, resp.
\smallskip

The \so{type} of a Hermitian   matrix $\,\mscrC\in {\rm M}_{n}(\K)$ is by definition the type a form with    $\mscrC$ as the  Gram's matrix with respect to an (arbitrary) $\K$-basis of $\K^{n}$. The  matrix analog of the Sylvester's Law of Inertia\,\ref{thm:3.4} is the following\,:
\end{definition}

\begin{corollary}\label{cor:3.6}
Let  $\Phi $ be a Hermitian   form on the  $n$-dimensional $\K$-vector space $V$ with $\K$-basis  $\mbcalx=\{x_1,\ldots ,x_n\}$.
Then $\Phi$ is of type $(p,q)$ if and only if  the Gram's matrix $\mscrG_{\Phi}(\mbcalx)$ is congruent to the matrix $\mscrE_n^{p,q}$, {\rm i.\,e.} there exists an invertible matrix  $\mscrA\in {\rm GL}_n\,(\K )$ such that $\mscrG_{\Phi}({\mbcalx}) =\,^{\rm t}\mscrA\mscrE_n^{p,q}\,\overline{\mscrA}\,$. Two Hermitian   matrices $\mscrC$ and $\mscrC'\in {\rm M}_{n}(\K)$ have the same type if and only if they are congruent. In~particular, a Hermitian  matrix $\mscrC\in {\rm M}_{n}(\K)$ have type $(p,q)$ if and only if $\mscrC$  is congruent to the matrix $\mscrE_{n}^{p,q}$.
\end{corollary}

If $\K\!=\!K$ (real closed), then one  can choose\,\footnote{\label{foot:3}\,Use the following observation\,: Let $V$ be an oriented vector space over a real-closed field $K$ of dimension $n\in\N^*$ and $\Phi$ be a Hermitian   form of type $(p,q)$ on $V$. Then there exists an orientation of  $V$ represented by a basis $x_1,\ldots ,x_n$ of $V\,$ such that the Gram's matrix of $\Phi$ is equal to the matrix ${\mscrE}_n^{p,q}$.}\,
 $\mscrA\!\in\!{\rm GL}_{n}^{+}(K )$, i.\,e. $\Det\mscrA>0$.  In the situation of  Corollary\,\ref{cor:3.6}, if $\Phi$ is non-degenerate, i.\,e. if  $p+q\!=\!n$, then ${\rm Det}\,{\mscrG}_{\Phi}({\mbcalx}) =(-1)^{q}|\,{\rm Det}\, {\mscrA}|^2$, i.\,e.
${\rm Sign}\, ({\rm Det}\,{\mscrG}_{\Phi}({\mbcalx}))\!=\!(-1)^{q}\!$. {\it Therefore, the sign of the Gram's determinant ${\rm Det}\,{\mscrG}_{\Phi}({\mbcalx})\,$ determines the parity of  $q$.} From this the following useful criterion for the determination of the type follows\,:

\begin{mypar}\label{thm:3.7} {\sc Hurwitz's Criterion}\,
{\it Let  $\Phi $ be a Hermitian  form on the  $n$-dimensional
$\K$-vector space $V$ with basis $\mbcalx=\{x_1,\ldots ,x_n\}$.
Suppose that the principal minors
\vspace*{-1mm}
\begin{align*}
D_{0}:=1\ \quad  {and} \quad \ D_i:=\left|\,\begin{matrix}
\Phi( x_1,x_1)&\!\!\cdots &\!\!\Phi( x_1,x_i) \cr
\vdots&\!\!\ddots&\!\!\vdots \cr
\Phi(x_i,x_1) &\!\!\cdots &\!\! \Phi( x_i,x_i) \cr
\end{matrix}\,\right|\,,\qquad i=1,\ldots ,n,
\vspace*{-2mm}
\end{align*}
of the Gram's matrix $\mscrG_{\Phi}({\mbcalx})=(\Phi (x_i,x_j))\in {\rm M}_n\,(K)$ of $\Phi$
with respect to the  basis $\mbcalx$
are all $\neq 0$. Then the type of  $\Phi$ is $(n-q,q)$, where $q$ is the number of
sign changes\,\footnote{\label{foot:4}\,Recall that we say that in a sequence $a_0,\ldots ,a_n$ of non-zero real numbers \so{changes the sign at the $i$-th place} if  $\,0\le i< n\,$ and $\,a_i \,a_{i+1}< 0$.\,---\,For an arbitrary sequence of real numbers $b_0,\ldots ,b_m$ by a \so{change of signs} means a change of signs in the sequence obtained by removing the zeros from the original sequence.}
in the sequence  $1=D_0,D_1,\ldots ,D_n={\rm Det}\,{\mscrG_{\Phi}({\mbcalx})}$.}
\end{mypar}
\smallskip

\begin{corollary}\label{cor:3.8}\,
Let  $\Phi $ be a Hermitian  form on the  $n$-dimensional
$\K$-vector space $V$ with basis $\mbcalx=\{x_1,\ldots ,x_n\}$  and let
\vspace*{-3mm}
\begin{align*}
D_{0}:=1\ \quad  {and} \quad \  D_i:=\left|\,\begin{matrix}
\Phi( x_1,x_1)&\!\!\cdots &\!\!\Phi( x_1,x_i) \cr
\vdots&\!\!\ddots&\!\!\vdots \cr
\Phi(x_i,x_1) &\!\!\cdots &\!\! \Phi( x_i,x_i) \cr
\end{matrix}\,\right|\,,\qquad i=1,\ldots ,n,
\vspace*{-2mm}
\end{align*}
be the principal minors of the Gram's matrix $\mscrG_{\Phi}(\mbcalx)$.  Then\,:
 \vskip3pt

 $(1)$\, $\Phi$ is positive  definite if and only if   $\,D_{i}>0$\,    for all $\,i=1,\ldots, n$.
\vskip2pt

($2)$\,   $\Phi$ is negative  definite if and only  if $(-1)^{i}\,D_{i}>0\,$ for all $i=1,\ldots, n$,\, {\rm i.\,e.}  at every position in the sequence $D_0, D_1,\ldots ,D_n$ there is a sign change.
\end{corollary}
\smallskip

\begin{example}\label{ex:3.9}\,
Let $\{v_1,v_2\}$ be a basis of the  2-dimensional  $K $-vector space $V$. For the symmetric bilinear form  $\Phi \!=\!\langle  -,- \rangle $ on  $V$. Let  $D_1\!=\!\langle v_1,v_1\rangle $ and
$D_2\!=\!
{\rm Det}\begin{pmatrix}
\langle v_1,v_1\rangle & \langle v_1,v_2\rangle \cr
\langle v_2,v_1\rangle & \langle v_2,v_2\rangle \cr
\end{pmatrix}\!=\!
\langle v_1,v_1\rangle\, \langle v_2,v_2\rangle -|\langle v_1,v_2\rangle |^2$.
Then the following table shows the dependence of the ${\rm sign }\,D_{1}$,  ${\rm sign }\, D_{2}$ and the  type of $\Phi$\,:\\[1mm]
\hspace*{10mm} {\footnotesize $\setbox\strutbox=\hbox{\vrule height 10pt depth 4.5pt width 0pt}
\offinterlineskip\tabskip=0pt
\vbox{\halign{\strut$\ #$ &#\vrule &\ #\ &#\vrule &\ #\ &#\vrule &\ #\ &#\vrule
&\ #\ &#\vrule
&\ #\ &#\vrule
&\ #\ &#\vrule
&\ #\ &#\vrule
&\ #\ & \ #\vrule
& \ # & \ # \vrule
& \ #  & \ # \vrule
&\ #\ \cr
D_1&&+&&+&&--&&--&&+&&--&&0&&0  && 0 && 0 \cr
\noalign{\hrule}
D_2&&+&&--&&+&&--&&0&&0&&--&&0 && 0 && 0 \cr
\noalign{\hrule}
\langle v_{1}, v_{2}\rangle &&   &&   &&   &&    &&   &&   &&   &&  >\,0  && <\,0  && 0 \cr
\noalign{\hrule}
{\rm Type}&& \ (2,0) \ && \ (1,1) \ && \ (0,2) \ && \ (1,1) \ && \ (1,0) \ && \ (0,1) \ && \ (1,1) \ && \ (1,0)\\ && \ (0,1) \ && \ (0,0) \ \cr}}\ \,.$ } \\[1mm]
Note that the case   $D_1=0$, $D_2>0$ is not possible.
\end{example}
\smallskip

\begin{example}\label{ex:3.10}
Let $z\in\C_{K}\smallsetminus K$, $\pi:=(X-z)(X-\overline{z}) \in K[X]$, $A:=K[X]/\langle \pi\rangle]:= K[x]$, where $x$ is the image of $X$ modulo $\langle \pi\rangle$. Further, let $H\in K[X]$, $H\not\in \langle \pi \rangle$,  $h=h(x)\in A$ be the image of $H$ in $A$ and let $\Phi_{h}:A\times A \longrightarrow K$ be the symmetric bilinear form defined by $\Phi_{h}(f,g) =\tr_{K}^{A}(hfg)$, $\,f$, $g\in A$. Then the Gram's matrix of $\Phi_{h}$ with respect to the basis $\{1,x\}$
$$\mscrG_{\Phi_{h}}(1, x) =
\begin{pmatrix}
h(z) +h(\overline{z}) &  h(z)\cdot z  +h(\overline{z})\cdot \overline{z} \cr
h(z)\cdot z  +h(\overline{z})\cdot \overline{z}  & h(z)\cdot z^{2}  +h(\overline{z})\cdot \overline{z}^{2} \cr
\end{pmatrix} \in {\rm M}_{2}(K)
$$
is a symmetric matrix with $\,D_{1}= h(z)+h(\overline{z})=2\,{\rm Re}\,h(z)\,$ and $\,D_{2}=\Det\,\mscrG_{\Phi_{h}}(1, x) = h(z)\, h(\overline{z}) \,(z-\overline{z})^{2}= -4\,|h(z)|^{2} ({\rm Re}\, z\,)^{2}<0$. Therefore, by the table in Example\,\ref{ex:3.10},  the type of $\Phi_{h}$ is $(1,1)$
\end{example}
\smallskip

The type of a Hermitian form on a finite dimensional vector space $V$ over $\C_{K}$ can also be determined by using the eigenvalues of the Gram's matrix, see Theorem\,\ref{thm:3.12} below. Usual proofs given in the standard text books of this fact uses the \textit{Principal Axis Theorem for self-adjoint operators} (also known ans the {\it Spectral Theorem}).  We give here a direct proof using the following interesting Lemma\,\ref{lem:3.11}\,:
\smallskip

\begin{lemma}\label{lem:3.11}$\!\!\!\!$ Let $K$ be a real closed field with notation as in {\rm \ref{mypar:3.1}},    $V$  an $n$-di\-mensional $\C_{K}$-vector space  $V$ with a positive  definite complex-Hermitian  form $\Phi$ on $V$ and let $f:V\rightarrow V$. Then there exists an orthonormal basis $\mbcalx\!=\!(x_1,\ldots ,x_n)$  of $V$
such that the matrix $\mscrM^{\mbcalx}_{\mbcalx}(f)$ of $f$ with respect to $\mbcalx$ is an upper triangular matrix.
\end{lemma}
\smallskip

\begin{theorem}\label{thm:3.12}
Let $K$ be a real closed field with notation as in {\rm \ref{mypar:3.1}} and let  ${\mscrC}\in \M_{n}(\C_{K})$ be a Hermitian  matrix.  Then all the eigenvalues of ${\mscrC}$ are in $K$ and  ${\mscrC}$ is of type $(p,q)$, where  $p$ is the number of positive eigenvalues and $q$ is the number of negative eigenvalues of  ${\mscrC}$ counted with their multiplicities in the characteristic polynomial $\chi_{\mscrC}$ of ${\mscrC}$.
\end{theorem}
\smallskip

\begin{remark}\label{rem:3.13}
The proof of the above Theorem\,\ref{thm:3.12} also  shows that every Hermitian matrix in $\M_{n}(\C_{K})$ is diagonalizable (even with respect to an orthonormal basis of $\C_{K}^{n}$).
\end{remark}

\begin{corollary}\label{cor:3.14}
Let $K$ be a real closed field with notation as in {\rm \ref{mypar:3.1}} and let ${\mscrC}\in \M_{n}(\C_{K})$ be a Hermitian  matrix. Then the characteristic polynomial  $\chi_{\mscrC}=c_{0}+c_{1}X+\cdots +c_{n-1}X^{n-1}+X^{n}\in K[X]$  and  ${\mscrC}$ is of type $(p,q)$,  where  $p$ is the number of sign changes in the sequence $c_{0}, c_{1}, \ldots , c_{n-1}, c_{n}:=1$ and $q$ is the
number of sign changes in the sequence $c_{0}, -c_{1}, \ldots , (-1)^{n-1}c_{n-1}, (-1)^{n}c_{n}=(-1)^{n}$. If  $c_{0}= c_{1}=\cdots = c_{r-1}=0$ and $c_{r}\neq 0$, then $p+q=n-r$.
\end{corollary}

\begin{proof}
Note that, since all the eigenvalues of ${\mscrC}$ are real by Theorem\,\ref{thm:3.12}, indeed $\chi_{\mscrA}\in K[X]$. The assertion is immediate from Theorem\,\ref{thm:3.12} and the following classical theorem of Descartes\,:
\end{proof}

\begin{theorem}\label{thm:3.15}\so{\rm (Descartes' Rule of Signs)}\,
Let $K$ be a real closed field with notation as in {\rm \ref{mypar:3.1}} and let
$f=a_{0}+a_{1}X+\cdots +a_{n-1}X^{n-1}+a_{n}X^{n}\in K[X]$, $a_{n}\neq 0$  be a polynomial of degree $n$. Further, let
${\rm V}_{+}$,  \hbox{$(${\it resp}. ${\rm V}_{-})$} denote  the number of sign changes in the sequence $a_{0}, a_{1}, \ldots , a_{n-1}, a_{n}$ {\rm (}resp. in the sequence $a_{0}, -a_{1}, \ldots , (-1)^{n-1}a_{n-1},(-1)^{n} a_{n})$  and  $N_{+} $ $($resp. $N_{-})$ denote the number of positive $($resp. negative$)$ zeros of $f$ $($each zero of $f$ is counted with its multi\-plicity$)$. Then there exist natural numbers $r_{+}$ and $r_{-}\in\N$ such that
$N_{+} ={\rm V}_{+} -2r_{+}$ and $N_{-} ={\rm V}_{-} -2 r_{-}$. Moreover, if all zeros of $f$ belong to $K$, {\rm i.\,e.} if $f$ splits into linear factors in $K[X]$, then $N_{+} ={\rm V}_{+}$ and $N_{-} ={\rm V}_{-}$.
\end{theorem}
\medskip

We now recall (from \cite{BS}) that  ``being of type $(p,q)$'' is an open property (with respect to the
\textit{strong topology\,}\footnote{\label{foot:5}\,{\bf Strong topology}\, Let $K$ be a real closed field (see\,Footnote\,\ref{foot:2}). Then
$K$ is equipped with the \so{order topology} which is determined by the base of the open intervals $]\,a,b\,[$, $a,b\in K$, $a<b$. The $K$-vector spaces $K^{n}$, $n\in\N$, are endowed with the \so{product topology} (with the  base given by the open cuboids $]\,a_{1},b_{1}\,[\times\cdots\times]\,a_{n},b_{n}\,[$, $a_{i}<b_{i}$, $i=1, \ldots , n$).  With the ordered and product topology, the addition, the multiplication and the inverse are continuous functions on $K\times K$ and $K^{\times}=K\setminus\{0\}$, respectively.
 Further, polynomial functions  (resp. rational functions $F/G$, $F$, $G\in K[X_{1}, \ldots , X_{n}]$, $G\neq 0$),  in $n$ variables are continuous $K$-valued functions on $K^{n}$ (resp. on $K^{n}\smallsetminus {\rm V}_{K}(G)$, where ${\rm V}_{K}(G):=\{a\in K\mid G(a)=0\}$ is a  zero set of the denominator $G$ which is closed in $K^{n}$. \\[2pt]
 The product topology on $K^{n}$ transfers uniquely to every $n$-dimensional $K$-vector space by a $K$-linear isomorphism $f:V\rightarrow K^{n}\!\!$. Any other isomorphism $g:V\rightarrow K^{n}$ defines the same topology, since $gf^{-1}:K^{n}\rightarrow K^{n}$ and $(gf^{-1})^{-1}=fg^{-1}:K^{n}\rightarrow K^{n}$ are continuous (polynomial) maps. Therefore,  polynomial and rational functions are also defined on any finite
dimensional vector space $V$ by  an isomorphism $f:V\rightarrow K^{n}\!\!$.
This topology on $V$ may be characterized as the smallest topology for which the $K$-linear functions $V\rightarrow K$ are continuous and  is called the \so{strong topology} on $V$, since it is stronger than  the \textit{Zariski topology}  on $V$  if $V\neq0$.}
)  which is an easy consequence of Hurwitz's Criterion\,\ref{thm:3.7}\,:
\smallskip

\begin{lemma}{\rm (cf.\,\cite[Lemma\,1.2]{BS}}\label{lem:3.16}
Let $K$ be a real closed field with notations as in {\rm \ref{mypar:3.1}} and $F_{ij}\in K[T]$  be polynomials such that $F_{ij}\!=\!F_{j\,i}$, $1\le i,j\le n$.  Suppose that  the bilinear form defined by the~symmetric matrix $(F_{ij}(s))_{1\le i,j\le n}\!\in\!\M_{n}(K)$ at  $s\!\in\!K$,    is non-degenerate, then there exists an $\varepsilon>0$ such that the type of the symmetric matrices $\,(F_{ij}(t))_{1\le i,j\le n}$ is the same for all $t\in]\,s-\varepsilon,s+\varepsilon\,[$. In~particular, for non-degenerate symmetric bilinear forms over $K$, ``being of type $(p,q)$'' is an open property.
\end{lemma}
\smallskip

We end this section by noting the following important Rigidity Theorem for symmetric bilinear forms  (see\,\cite{BS} which is proved by using Hurwitz's Criterion\,\ref{thm:3.7},  the above Lemma\,\ref{lem:3.16} and the \textit{Intermediate Value Theorem for polynomial functions}\,\footnote{\label{foot:6}\,{\bf Intermediate Value Theorem for polynomial functions}\,
{\it Let $K$ be a real closed field and  $F\in K[T]$ be  a polynomial with coefficients in $K$ such that $F(a)F(b)<0$ for some $a,b\in K$. Then $F$ has a zero in $[a,b]$.}  In other words, the values $F(t)$, $t\in[a,b]$, have the same sign if $F$ has no zero on $[a,b]$. In particular, every polynomial of odd degree has a zero in $K$. A field with this property is called a \so{$2$-field}.  Therefore, \emph{a real closed field is a $2$-field}. Furthermore, every monic polynomial $F$ over a real closed field $K$ has a positive zero in $K$ if $F(0)<0$ (since $F(x)>0$ for ``large'' $x$).
}.
\smallskip

\begin{mypar}{\sc Rigidity Theorem for Quadratic Forms}\label{thm:3.16}\,{\rm (cf.\,\cite[1.3]{BS})}\,
Let $K$ be a real closed field with notations as in {\rm \ref{mypar:3.1}} and let  $R_{ij}(t)=R_{ij}(t_{1},\ldots , t_{n})$, $1\le i\,,\,j\le n$, be rational functions on a line-connected\,\footnote{\label{foot:7}\,{\bf Line-connected subsets}\, Let $V$ be a vector space over a real closed field $K$.
For two points $x,y\in V$, the subset $[x, y]=[y,x]:=\{(1-t)x+ty\mid t\in K,0\le t\le1\}\subseteq V$ is called
 the (closed) \so{line-segment}  connecting $x$ and $y$.  For $x_{0}, \ldots , x_{r}\in V$, $r\ge 1$, the subset  $[x_{1}, \ldots , x_{r}]:=\cup_{i=1}^{r}[x_{i-1},x_{i}]$ is called the \so{broken line from $x_{0}$ to $x_{r}$}.
 A subset $V'\subseteq V$ is called \so{line-connected} if for any two points $x,y\in V'$ there is a broken line from $x$ to $y$ which lies entirely in $V'$.
 Note that, if $K=\R$ and  $U\subseteq V$ is open (in the strong topology, see\,Footnote\,\ref{foot:5}), then the notion ``line-connected'' is equivalent to the topological notion of ``connected''. The only topologically connected subspaces of $K=\Q$  are the singletons. If $V$ is a line, i.\,e.\ $1$-dimensional, and if $x\in V$, then $V\smallsetminus\{x\}$ is not line-connected. However, if $\,\Dim_{K}\!V\ge 2$, then $V\smallsetminus\{x\}$ is always
 line-connected\,: If $u,w\in V\smallsetminus\{x\}$ are arbitrary points, there is always a point $v\in V\setminus\{x\}$ such that $[u,v,w]\subseteq V\smallsetminus\{x\}$.}
subset $U\subseteq K^{n}$ such that the matrices $\mscrR(t)=(R_{ij}(t))_{1\le i,j\le n}\in\M_{n}(K)$, $t\in U$, are symmetric, {\rm i.\,e.}  $R_{ij}=R_{ji}$ for all $1\leq i\,,\,j\leq n$ with a   $\Det\,\mscrR(t)\neq0$ for all $t\in U$. Then all the matrices $\mscrR(t)\in\M_{n}(K)$, $t\in U$, have the same type $\,(p,q)$, or equivalently, the same signature $p-q$.
\end{mypar}


\section{Trace forms and Rational points}\label{sec:TraceForm}

In this section, we recall the results from \cite{BS} (based on the talk of Prof. U.\,Storch at IIT Bombay in November\,2009) on trace forms, their invariants such as  rank, type, signature and their relations with the number of rational points of a finite algebra $A$ over a real closed field. For detailed proofs of these results the reader is recommended to see \cite[$\S$\,3]{BS}.


\begin{mypar}\label{mypar:4.1}{\sc Preliminaries\,}\,
In this subsection, we recall the basic concepts from elementary commutative algebra (see  \cite{AM},  \cite{Kunz1980} and \cite{PatilStorch2010}).
\medskip

Let $A$ be an arbitrary commutative ring (with unity).
The set  $\Spec A$ (resp. $\Spm A$) of prime (resp. maximal) ideals in $A$ is called the {\it prime} (resp. {\it maximal}) {\it spectrum} of $A$. The nil-radical $\mathfrak{n}_{A}:= \sqrt{0}=\cap_{\,\gothp\in\Spec A}\,\gothp$ is the intersection of all prime ideals in~$A$.
More generally,  (\so{formal Nullstellensatz})\, $\sqrt{\mscrA }= \cap_{\,\gothp\in\Spec A}\,\{\gothp \mid \mscrA\subseteq \gothp\}$ for every ideal $\mscrA$  in $A$, see \cite{AM},  \cite{Kunz1980}. \\[1mm]
 The intersection $\displaystyle\mscrM_{A}:=\cap_{\,\mscrM\in\Spm A}\,\mscrM$ of all maximal ideals in $A$ is called the {\it Jacobson radical} of $A$.

\begin{alist}
\medskip

\item {\bf The ${\bf K}$-Spectrum and  the set of ${\bf K}$-rational points of a ${\bf K}$-algebra}\, (see \cite{PatilStorch2010})
Let $K$ be a field. Using the universal property of the polynomial algebra $K[X_{1},\ldots , X_{n}]$,   the affine space $K^{n}$ can be identified with the set of $K$-algebra homomorphisms
$\Hom_{\Kalg}(K[X_{1},\ldots,X_{n}]\,,K)$ by identifying the point  $a =
(a_{1},\ldots,a_{n}) \in K^{n}$ with the substitution homomorphism
$\xi_{a}: K[X_{1},\ldots,X_{n}] \rightarrow K$, $X_{i} \mapsto a_{i}$, whose kernel  $\Ker\,\xi_{a}$ is
the maximal ideal $\mathfrak{m}_{a} = \langle X_{1}-a_{1},\ldots,X_{n}-a_{n}\rangle $ in $K[X_{1},\ldots,X_{n}]$. Moreover, every maximal ideal $\mathfrak{m}$ in $K[X_{1},\ldots,X_{n}]$ with $K[X_{1},\ldots,X_{n}]/\mathfrak{m} =K$ is of the type $\mathfrak{m}_{a}$ for a unique point $a=(a_{1},\ldots , a_{n})\in K^{n}$; the component $a_{\,i}$ is determined by the congruence $X_{i} \equiv~a_{i}~ {\rm mod}~ \mathfrak{m}$.
\smallskip

\noindent The subset $\,\KSpec  K[X_{1},\ldots,X_{n}]:=\{\mathfrak{m}_{a} \mid a\in K^{n}\}$ of $\Spm K[X_{1},\ldots,X_{n}]$ is called the {\it $K$-spectrum} of
$K[X_{1},\ldots,X_{n}]$.  We have the identifications\,:
\vspace*{-0.75mm}
\begin{align*}
\enskip\enskip K^{n}  & \xlongleftrightarrow{\hspace*{13mm}}  \enskip \Hom_{\Kalg}(K[X_{1},\ldots,X_{n}]\,,K) \enskip  \xlongleftrightarrow{\hspace*{13mm}}  \enskip \KSpec K[X_{1},\ldots,X_{n}]\,,\\
\enskip a  \enskip \enskip & \xlongleftrightarrow{\hspace*{35mm}} \enskip   \xi_{\,a}  \enskip \xlongleftrightarrow{\hspace*{35mm}} \enskip  \mathfrak{m}_{a}=\Ker
\xi_{\,a}\,.
\vspace*{-2mm}
\end{align*}
More generally, for any $K$-algebra $A$, the map \\[1mm]
\hspace*{20mm} $\,\Hom_{\Kalg}(A\,,K)\longrightarrow
\{\mscrM\in\Spm A\mid A/\mscrM =K\}$, $\xi\longmapsto \Ker\xi$, \\[1mm]
is bijective. Therefore we make the following definition\,:
\smallskip

\noindent For any $K$-algebra $A$ of finite type, the subset
$\,\KSpec A:= \{\mscrM\in\Spm A\mid A/\mscrM =K\}\,$ is
called the $K$-\,\so{spectrum} of $A$ and is denoted by $\KSpec A$. \\[1mm]
Further, if $A\iso K[X_{1}, \ldots , X_{n}]/\mathfrak{A}$ is a representation of the finite $K$-algebra $A$, then the $K$-algebraic set ${\rm V}_{K}(\mathfrak{A} :=\{a\in K^{n}\mid F(a)=0\,  \  \hbox{for all} \ \,F\in\mathfrak{A}\}$ defined by the ideal $\mathfrak{A}$  is called the set of $K$-\,\so{rational points} of $A$.\\[1mm]
Under the above bijective maps, we have the identification ${\rm V}_{K}(\mathfrak{A})= \Hom_{\Kalg}(A\,,K) = \KSpec A$.
For example, since $\C$ is an algebraically closed field, $\Spm \C[X] = \CSpec \C[X]$, but $\RSpec\R[X] \subsetneqq  \Spm\R[X]$.  In fact, the maximal ideal $\mathfrak{m}:=\langle X^{2}+1\rangle \in\Spm\R[X]$  does not belong to $\RSpec\R[X]$. More generally, a field $K$ is algebraically closed if and only if $\Spm K[X] = \KSpec K[X]$,  see\, \cite{AM}\,, \cite{Kunz1980} or  \cite[Theorem\,2.10\,,\,HNS\,3]{GPV2018}\,)
\smallskip

\item
{\bf Local components of a finite algebra}\,
Let $A$ be a finite algebra over a field $K$, i.\,e. $A$ finite dimensional as a $K$-vector space of dimension ${\rm Dim}_{K} A$.  Then  $\Spm A=\Spec A$ (since any finite $K$-algebra which is an integral domain is already a field). Moreover, from  the Chinese Remainder Theorem, it follows immediately that $\Spm A$ \textit{is a finite set}.
In particular,  $\#\Spm A \leq \Dim_{K}\!A$ and equality holds if and only if $A$ is isomorphic to the product $K$-algebra $K^{\,\Dim_{K}\!A}$\,). \\[1mm]
Further, let $\Spm A=\{\mscrM_{1}, \ldots , \mscrM_{r}\}$. Then  the unit group $A^{\times}$ of $A$ is $\,A\!\smallsetminus\,\bigcup_{\,i=1}^{\,r}\mscrM_{i}$ and the canonical homomorphism $A\longrightarrow \prod_{\,i=1}^{\,r}A_{\mscrM_{i}}$ is injective (where $A_{\gothp}$ denotes the localization of $A$ at a prime ideal $\gothp\in \Spec A$). In our special case, it is also surjective and hence an isomorphism, cf. \cite[Corollary\,55.16]{SchejaStorch1988}.
Therefore, \emph{$A$ is the direct product of the local finite $K$-algebras} $A_{i}:=A_{\mscrM_{i}}$, $i=1, \ldots ,  r$, which are called the \so{local components} of $A$. Furthermore, we have\,:\,
$\,\Dim_{K}A=\sum_{\,i=1}^{r}\Dim_{K}A_{i}=\sum_{i=1}^{r}\ell(A_{i})\cdot[K_{i}:K]$,
where, for $\,i= 1, \ldots , r$, $K_{i}=A/\mscrM_{i}$ is the residue class field of $A$ at $\mscrM_{i}$ and $\ell(A_{i})$ the (finite) length of $A_{i}$, i.\,e. the length $\ell$ of a composition series $0=\mscrA_{0}\subsetneq\mscrA_{1}\subsetneq\cdots\subsetneq\mscrA_{\ell}=A_{i}$ with $\mscrA_{j+1}/\mscrA_{j}\cong A/\mscrM_{i}$, $j=1, \ldots , \ell-1$.\\[1mm]
For example, if $K$ is a $2$-field, then $\,[K_{i}:K]\,$ is even if $K_{i}$ is a non-trivial field extension of $K$ and, in particular, $\KSpec A\neq\emptyset$ {\it if $\,\Dim_{K}A$ is odd}.
\smallskip

\noindent
Further,  $\mscrM_{A}=\mscrM_{1}\cap\cdots\cap\mscrM_{r} = \cap_{\gothp\in\Spec A}\, \gothp = \gothn_{A}$ and  \textit{ $\mscrM_{A}=\gothn_{A}=0$, {\rm i.\,e.}  $A$ is reduced if and only if $A=K_{1}\times \cdots\times  K_{r}$ is the product of its residue class fields}. Moreover, if  all the field extensions $K_{i}$ of $K$ are separable, then $A$ is called a (finite) \so{separable} $K$-algebra.

\end{alist}
\end{mypar}
\smallskip

\begin{mypar}\label{mypar:4.2}{\sc The trace form}\,
Let $A$ be a finite algebra over the field $K$. The \so{trace form}  on $A$ over $K$ is the
symmetric $K$-bilinear form $\tr:=\tr_{K}^{A}:A\times A \rightarrow K$, $(f,g)\mapsto \tr_{K}^{A}(fg)$ on $A$. It is a classical tool used to study the $K$-algebra $A$. \\[1mm]
The decomposition of
$A=A_{1}\times\cdots\times A_{r}$ into its local components (cf.\ \ref{mypar:4.1}\,(c)) yields the orthogonal decomposition (see\,Decompsition Theorem\,\ref{decthm:2.13}) \\[0.75mm]
\hspace*{42mm} $\,\tr_{K}^{A}=\tr_{K}^{A_{1}}\operp\cdots\operp\tr_{K}^{A_{r}}\,$ \\[0.5mm]
of the trace form. The degeneration space $A^{^{_{\bot}}}\!\!\! =\! A^{^{_{\bot_{\tr}}}}\!\!\! =\! \{f\in A\mid \tr(Af)\!\! =\! 0\}$ is an ideal in $A$.
\end{mypar}
\smallskip

\begin{lemma}\label{lem:4.3}{\rm (cf.\,\cite[Lemma\,3.1]{BS})}\,
Let $A$ be a finite algebra over an arbitrary field $K$
 and let $A^{^{_{\bot}}}$ be the degeneration space of the trace form $\tr_{K}^{A}$. Then radical $\mscrM_{A}=\gothn_{A}\subseteq A^{^{_{\bot}}}$. Moreover, equality holds  if and only if all the residue class fields of $A$ are separable over $K$, {\rm i.\,e.} if and only if the reduction $A_{\red} = A/\gothn_{A}$ is a separable $K$-algebra. ---\,In particular, the trace form is non-degenerate if and only if $A$ is a separable $K$-algebra.
\end{lemma}
 \smallskip

\begin{corollary}\label{cor:4.4}
Let $A$ be a finite separable algebra over an arbitrary field $K$. Then

\vspace*{-1.5mm}
\hspace*{25mm} $\,\displaystyle \rank\,\tr_{K}^{A}=\Dim_{K}(A/\mscrM_{A})=\sum_{i=1}^{r}\,[K_{i}:K]$.
\vspace*{-2mm}

Moreover, if $K$ is an ordered field, then\,:

\vspace*{-1.5mm}
\hspace*{20mm}$\displaystyle\type\,\tr_{K}^{A}=\sum_{i=1}^{r}\type \tr_{K}^{K_{i}}\,$\  and \  $\,\displaystyle \sign \tr_{K}^{A}=
\sum_{i=1}^{r}\sign\,\tr_{K}^{K_{i}}$.
\vspace*{-1mm}
\end{corollary}
\smallskip

Now, we state the following important and classical criterion for the existence of $K$-rational points for real closed fields which is proved in  \cite{BS}.
\smallskip

\begin{theorem}{\rm (cf.\,\cite[Theorem\,3.2]{BS})}\label{thm:4.5}\,
Let $A$ be a finite  algebra over a real closed field $K$. Then\,:
\vspace*{1mm}

\hspace*{40mm} $\, \sign\,\tr_{K}^{A}=\#\,\KSpec A$.
\vspace*{1mm}

In particular, $K$ is a residue class field of $A$ if and only if $\,\sign\,\,\tr_{K}^{A}\neq0$.
\end{theorem}
\smallskip

\begin{example}\label{ex:4.6}
Let $K$ be a real closed field. Then there is a unique (up to isomorphism) non-trivial finite field extension $L\,\vert\,K$, namely, the quadratic field $L=\C_{K}=K[\,{\rm i}\,]$  with ${\rm i}^{2}=-1$, of \so{complex numbers over} $K$ (which is the  algebraic closure of $K$).   The Gram's matrix of the trace form ${\tr}^{\C _{K}}_{K}$  of $\C_{K}$ over $K$ with respect to the basis $1$, ${\rm i}\,$ is the matrix
\vspace*{-1mm}
$$\left(\begin{array}{cc}
\tr\,(1)&\tr\,({\rm i})\\
\tr\,({\rm i})&\tr\,(-1)
\end{array}\right)\,=\,
\left(\begin{array}{cr}
2&0\\
0&-2
\end{array}\right)\,.$$
\vspace*{-1mm}
Therefore  $\,\type\,\tr_{K}^{\C _{K}}=(1,1)\,$ and $\,\sign\,\tr_{K}^{\C _{K}}=0$.
\end{example}
\smallskip

\begin{corollary}\label{cor:4.7}
Let $A$ be a finite algebra over a real closed field $K$. Then the trace form $\tr_{K}^{A}$ is positive definite if and only if $A$ is separable over $K$ and $A$ splits over $K$, {\rm i.\,e.} there exists an isomorphism of $K$-algebras  $A\iso K^{\,\Dim_{K}A}$.
\end{corollary}
\smallskip

\begin{corollary}\label{cor:4.8}
Let $K$ be a real closed field and $f\in K[X]$ be a monic polynomial. Then all zeros of $f$ (in $\overline{K}$) belong to $K$ and are simple if and only if the trace form $\tr_{K}^{A}$ of the $K$-algebra $A:=K[X]/\langle f\rangle$ is positive definite.
\end{corollary}
\smallskip

For a partial generalization (see Theorem\,\ref{thm:4.10} below) of Theorem\,\ref{thm:4.5} and applications one can also consider the following  more general trace forms\,:
\smallskip

\begin{mypar}\label{mypar:4.9}{\sc Generalized trace forms\,}\,
Let  \\[1mm]
\hspace*{25mm} ${\rm Sym}_{K}(V,K):=\{\Phi\in {\rm Mult}_{K}\,(V,K) \mid \Phi \ \hbox{is symmetric} \,\}$  \\[1mm]
be the  $K$-vector space of all symmetric bilinear forms on $V$ and
consider a $K$-linear embedding \\[2mm]
\hspace*{10mm} $\,E_{A|K}:=\Hom_{K}(A,K) \longrightarrow {\rm Sym}_{K}(V,K)$, $\,\alpha\longmapsto \Phi_{\alpha}:A\times A\to K\;, \, (f,g)\mapsto\alpha(fg)\,$. \\[2mm]
The elements of the image of this map are called \so{generalized trace forms} on $A$.
The $A$-module $E_{A\,\vert K}$ ( with the scalar multiplication $(g\alpha)(f):=\alpha(fg)$ for $\alpha\in E$, $g$, $f\in A$)  is  called the \so{dualizing module} of $A$.
Therefore\,:
\smallskip \\
{\bf  \ref{mypar:4.9}.1}\,  {\it   $\Phi_{\alpha}(f,g)=(g\alpha)(f)=(f\alpha)(g)$ and the degeneration space $A^{^{_{\bot_{\alpha}}}}$ of $\,\Phi _{\alpha}$  is the largest ideal of $A$ contained in $\Ker\,\alpha\,$.}
\smallskip \\
{\bf  \ref{mypar:4.9}.2}\, {\it Let $\overline{\alpha}:A/A^{^{_{\bot_{\alpha}}}}\to K$ be  the linear form on $\overline{A}:=A/A^{^{_{\bot_{\alpha}}}}$ induced by $\alpha$,  then
$\rank\Phi_{\alpha}=\rank\Phi_{\,\overline{\alpha}}$ and the induced  bilinear form $\Phi_{\,\overline{\alpha}}$ is non-degenerate on $\overline{A}$.}
\smallskip \\
{\bf  \ref{mypar:4.9}.3}\, {\it Moreover, if $K$ is an ordered field,  then}\,: \\[1mm]
\hspace*{30mm} \hbox{$\type\Phi_{\alpha}=\type\Phi_{\,\overline{\alpha}}\,$   and   $\,\sign\Phi_{\alpha}=\sign \Phi_{\,\overline{\alpha}}$.}
\medskip

For example, for  a fixed $h\in A$, the symmetric bilinear from $\Phi_{h}:A\times A \rightarrow K$, $(f, f')\mapsto \tr_{K}^{A}(hff')$ is the generalized trace from on $A$ with respect to the $K$-linear form $\lambda_{\,h}:A\rightarrow A$, $g\mapsto hg$.
\smallskip

We shall use these particular generalized trace forms on $A$ and the following partial generalization of the  Theorem\,\ref{thm:4.5}\,   in the proof of Theorem\,\ref{thm:5.5}.

\end{mypar}
\smallskip

\begin{theorem}{\rm (cf.\,\cite[Theorem\,3.4]{BS})}\label{thm:4.10}\,
Let $\alpha$ be a $K$-linear form on a finite  algebra $A$ over a
real closed field $K$. If $\,\sign\Phi_{\alpha}\neq0$, then $A$ has a $K$-rational point, {\rm i.\,e.}\,  $\KSpec A\neq\emptyset$.
\end{theorem}


\section{Counting rational points of 0-dimensional affine algebraic sets}\label{sec:Sec5}

In this section we will apply results from Section\,4 on trace forms to count the rational points  of finite affine algebraic sets over  real closed fields. Our method is a modern version of old results of Hermite and Sylvester who had used signatures of quadratic forms to count real zeros of polynomials in one variable, see\,\cite{Hermite1852}, \cite{Hermite1852a} and \cite{Sylvester1853}. We use elementary commutative  algebra  to treat multivariate versions of these problems.
\smallskip

\begin{mypar}\label{mypar:5.1}{\sc Notation, Assumptions and Consequences}\, Throughout this section, we use the following notation and assumptions and their consequences\,:
\smallskip

Let $K$ be a real closed field with notation as in \ref{mypar:3.1}.
For an ideal  $\mathfrak{A}\subseteq K[X_{1},\ldots , X_{n}]$ in the polynomial ring
$K[X_{1},\ldots , X_{n}]$ over $K$,  let \\[1mm]
\hspace*{30mm} ${\rm V}_{K}(\mathfrak{A}):=\{a\in K^{n}\mid F(a)=0 \ \hbox{ for all } \ F\in\mathfrak{A}\}$\,
and  \\[1mm]
\hspace*{30mm} $\,{\rm V}_{\K}(\mathfrak{A}):=\{a\in \K^{n}\mid F(a)=0 \ \hbox{ for all } \ F\in\mathfrak{A}\}$.  \\[1mm]
be the affine algebraic set in $K^{n}$ and in $\K^{n}$ defined by $\mathfrak{A}$, respectively.
\smallskip

Polynomials in $K[X_{1},\ldots , X_{n}]$ are denoted by capital letters $F$, $G$, $H$, $\ldots$ and their images in the residue class $K$-algebra $A:=K[X_{1},\ldots , X_{n}]/\mathfrak{A}$ are denoted by small letters $f$, $g$, $h$, $\ldots$\,.
\smallskip

Every  element $f\in A$ defines a (regular or polynomial) function on ${\rm V}_{K}(\mathfrak{A})$, namely  $\,f:{\rm V}_{K}(\mathfrak{A}) \longrightarrow K$, $a\longmapsto f(a)$. Further, if   $f$, $g\in A$, then, clearly\,: \\[1mm]
\hspace*{9mm} $f=g$ on ${\rm V}_{K}(\mathfrak{A})\iff f=g$ in $A \iff \, F \equiv G~{\rm (\,mod\,} \mathfrak{A}\,)$, i.\,e. $F-G\in\mathfrak{A}$.
\medskip

We assume that the residue class $K$-algebra $\,A:=K[X_{1},\ldots , X_{n}]/\mathfrak{A}$ is finite dimensional $K$-vector space, or equivalently,  the affine algebraic set  ${\rm V}_{K}(\mathfrak{A})\subseteq K^{n}$ is a finite set.
These assumptions are equivalent with the conditions\,:    the $\K$-algebra $\K\otimes_{K}A= A_{\K} = \K[X_{1},\ldots , X_{n}]/\langle \mathfrak{A}\rangle $ is finite dimensional over $\K$, or  equivalently,  the affine algebraic subset  ${\rm V}_{\K}(\mathfrak{A}) \subseteq \K^{n}$  is a finite set.
\smallskip

Further, since $\mathfrak{A} \subseteq K[X_{1},\ldots , X_{n}]$, it follows that if ${\bf a}\!\in\! {\rm V}_{\K}(\mathfrak{A})$, then its conjuga\-te $\overline{{\bf a}}\!\in\! {\rm V}_{\K}(\mathfrak{A})$, too. Therefore, since ${\rm V}_{K}(\mathfrak{A}) \subseteq {\rm V}_{\K}(\mathfrak{A})$, \hbox{renumbering we assume that\,:}
\smallskip

{\bf 5.1.a\,}\, ${\rm V}_{K}(\mathfrak{A})\!=\!\{{\bf a}_{1},\ldots , {\bf a}_{r}\}\!\subseteq\!
{\rm V}_{\K}(\mathfrak{A})\!=\!\{{\bf a}_{1},\ldots , {\bf a}_{r}\,,\, {\bf a}_{r+1}\,,\, \overline{{\bf a}}_{r+1}\,,\, \ldots\,,\, {\bf a}_{r+s}\,,\, \overline{{\bf a}}_{r+s}\}$, \\[1mm]
where  $\,r:=\#{\rm V}_{K}(\mathfrak{A})$, $r+s=\#\Spm A$ \hspace*{1mm} and
 \hspace*{1mm} $m:=r+2s=\Dim_{K}A=\Dim_{\K} A_{\K} =\#{\rm V}_{\K}(\mathfrak{A})$. \\[1mm]
Furthermore, since $K$ is a real closed field, ${\rm Char}\, K=0$, in~particular, $K$ is infinite  and hence by a linear change of coordinates (over $K$) (for instance, $Y_{i}=X_{i}$ for all $i=1,\ldots , n-1$ and $Y_{n}= X_{n}+\sum_{\,i=1}^{n-1}\, X_{i}\, t^{i}$ for suitable $t\in K$ avoiding finitely many), we may assume that ${\rm V}_{\K}(\mathfrak{A})$ is in \so{general $X_{n}$-position}, or the ideal $\mathfrak{A}$ in \so{general $X_{n}$-position}\, (The intention is to separate all zeros in an algebraic closure of $K$ by their last coordinate),\,  i.\,e.\,: \\[2mm]
{\bf 5.1.b\,}\, The $n$-th coordinates $a_{i\,n}$ of the points ${\bf a}_{i}\!=\! (a_{i1}, \ldots , a_{i\,n})\in \K^{n}\!\!$, $i\!=\!1,\ldots , m$ are all distinct. \\[1.5mm]
Note that ${\rm V}_{K}(\mathfrak{A})\!=\!{\rm V}_{\K}(\mathfrak{A})\cap K^{n}$ is the set of $K$-rational points of ${\rm V}_{\K}(\mathfrak{A})\!\iso\!\KSpec A_{\K}= \Spm A_{\K} =
\Spec A_{\K}$ (the first equality follows from Hilbert's Nullstellensatz, see\,\cite{Kunz1980} or  \cite[Theorem\,2.10, HNS\,3]{GPV2018}\,) and
${\rm V}_{K}(\mathfrak{A}) \iso \KSpec A\subseteq \Spm A\!=\!\Spec A$,  see\,\ref{mypar:4.1}\,(b). Further, since $A$ and $A_{\K}$ are reduced, the local components (see\,\ref{mypar:4.1}\,(c)) of  $A$ corresponding to the $K$-rational points ${\bf a}_{i}\!\in \!{\rm V}_{K}(\mathfrak{A})$, $i=1,\ldots , r$,  are isomorphic to $K$ and  corresponding to $\mscrM\in\Spm A \!\smallsetminus\!\KSpec A$    are  isomorphic to $\K$, but local components of  $A_{\K}$ corresponding to all the points ${\bf a}\in {\rm V}_{\K}(\mathfrak{A})$ are all isomorphic to $\K$. Therefore the explicit structures of the $K$-algebra $A$ and the $\K$-algebra $A_{\K}$ are determined by the algebra  isomorphisms which are defined by the substitutions\,: \\[1.5mm]
{\bf 5.1.c\,}\, \hspace*{2mm} $\,A \iso\enskip  K^{r}\times \K^{s}\,$, $h\mapsto \left(h\,(\,{\rm mod}\,\mscrM\,)\right)_{\mscrM\in \Spm A}$,  where $r$,  $s$ as in {\bf 5.1.a} \  and \ \\[1mm]
\hspace*{12mm} $\,A_{\K} \iso \enskip \K^{m}$, $f\mapsto \left(f({\bf a})\right)_{{\bf a}\in {\rm V}_{\K}(\mathfrak{A})}\,$,  where  $m:=r+2s$. \\[1.5mm]
Note that $m=\Dim_{K}A=\Dim_{\K} A_{\K} =\#{\rm V}_{\K}(\mathfrak{A})$.    \\[1mm]
Furthermore,   the following eigenvector theorem (see \cite[Ch.\,2, \S4, Theorem\,4.5]{CLO2005} which follows directly from 5.1.b\,:  \\[1.5mm]
{\bf 5.1.d\,}
For every $h\!\in\!A$,
the  eigenvalues of the $K$-linear map $\lambda_{h}:A\to A$, $f\mapsto h f$ are the values $h({\bf a}_{1}), \ldots , h({\bf a}_{r})$,  $\,h({\bf a}_{r+1}), h(\overline{{\bf a}}_{r+1})
\ldots , h({\bf a}_{r+s}), h(\overline{{\bf a}}_{r+s})$ of the function $h:{\rm V}_{\K}(\mathfrak{A}) \rightarrow \K$.\\[2mm]
For more accessible determination of the signature of the trace form $\, \tr^{A}_{K}$, we need a nice basis of $A$ over $K$. The following crucial key observation so-called Shape Lemma (see  \cite{CLO2005},  \cite{GianniMora1989} and \cite{KR2000})
guarantees a distinguished generating set  for a radical ideal $\mathfrak{A}$ in $K[X_{1},\ldots , X_{n}]$.
We give a  proof of the Shape Lemma by using the natural action of the Galois group $\Gal (\overline{K}\vert K)$ on
${\rm V}_{\overline{K}}(\mathfrak{A})$.
\end{mypar}
\smallskip

\begin{lemma}\label{lem:5.2}{\rm  (\,\so{Shape Lemma}\,)}\,
Let $K$ be an infinite perfect field and $\mathfrak{A}\subseteq K[X_{1},\ldots , X_{n}]$ be a radical ideal  and let   $A:= K[X_{1},\ldots , X_{n}]/\mathfrak{A}$ be  a finite dimensional $K$-vector space. With further  notation and assumptions as in {\rm \ref{mypar:5.1}}. There exist polynomials $g_{1},\ldots , g_{n-1},\, g_{n} \in K[T]\, ($where $T$ is  indeterminate over $K)$  with $g_{n}\neq 0$ square free of degree $m$, such that $\,\mathfrak{A}$ is generated by
$X_{1}-g_{1}(X_{n})\,, \ldots ,\, X_{n-1}-g_{n-1}(X_{n})$,  $g_{n}(X_{n})\,$.
In~particular,    $\,\mbcalx= \{1, x_{n}, \ldots , x_{n}^{m-1}\}$ is a $K$-basis of $A$, where $x_{n}$ is the image of $X_{n}$ in $A$.
\end{lemma}

\begin{proof}
Let $\overline{K}$ be an algebraic closure of $K$. Since $K$ is perfect, the field extension  $\overline{K}\vert K$ is a Galois extension. Let $\Gal (\overline{K}\vert K)$ be its Galois group.  \\[0.5mm]
Let ${\rm V}_{\overline{K}}(\mathfrak{A})\!:=\!\{a\in \overline{K}^{n}\!\mid F(a)\!=0\!\hbox{ for all } F\!\in\!\mathfrak{A} \}$. Then ${\rm V}_{\overline{K}}(\mathfrak{A})$ is a finite set by assum\-p\-tion on $\mathfrak{A}$ and the projection map $q:{\rm V}_{\overline{K}}(\mathfrak{A}) \rightarrow \overline{K}$, $(a_{1},\ldots , a_{n})\mapsto a_{n}\,$ is injective (by assumption (see 5.1.b)), (i.\,e. $q\,$ separates points in ${\rm V}_{\overline{K}}(\mathfrak{A})$). \\[0.75mm]
The Galois group ${\rm Gal}\,(\overline{K}\vert K)$ operates  on ${\rm V}_{\overline{K}}(\mathfrak{A})$ with the natural operation\,: \\[0.75mm]
\hspace*{8mm} $\,{\rm Gal}\,(\overline{K}\vert K) \times {\rm V}_{\overline{K}}(\mathfrak{A}) \longrightarrow {\rm V}_{\overline{K}}(\mathfrak{A})$,
$(\sigma, (a_{1},\ldots , a_{n}))\longmapsto (\sigma(a_{1}),\ldots , \sigma(a_{n}))$. \\[0.75mm]
Obviously, the image $q({\rm V}_{\overline{K}}(\mathfrak{A}))\!=\!W_{1}\uplus \cdots \uplus W_{\ell}$ is the union of  orbits of this operation and each orbit $W_{k}\!=\!{\rm V}_{\overline{K}}(\pi_{k})$ is the zero set of the irreducible polynomial $\pi_{k}\!\in\! K[T]$, $k\!=\!1,\ldots , \ell$,\, see\,\cite{Jacobson1989}\,  or\,  \cite[Ch.\,XI, \S93, 93.2]{SchejaStorch1988}.   Therefore, since $K$ is perfect, the polynomial $g _{n}:=\pi_{1}\cdots \pi_{\ell}\in K[T]$  is square free and   $q({\rm V}_{\overline{K}}(\mathfrak{A}))\!=\!{\rm V}_{\overline{K}}(g_{n})$, $\deg\, g _{n}\!=\!\# {\rm V}_{\overline{K}}(\mscrM)\!=\!m$. \\[1.5mm]
{\bf 5.2.a}\, For all $\,a_{n}\in  q({\rm V}_{\overline{K}}(\mathfrak{A}))$, there exist polynomials $\,g_{1},\ldots , g_{n-1}\in K[T]$ with   $\deg g_{i}<\deg g_{n}=m$ such that $(g_{1}(a_{n}), \ldots , g_{n-1}(a_{n}), a_{n})$ is the unique point lying over $a_{n}$.\\[1.5mm]
To prove 5.2.a, let  $a_{n}\in q({\rm V}_{\overline{K}}(\mathfrak{A}))$ and  $(a_{1}, \ldots , a_{n-1},\, a_{n})$ be the unique point lying over $a_{n}$.  We may assume that $a_{n}\in W_{1}=\{\sigma_{j}(a_{n}) \mid j=1, \ldots , d_{1},\, \sigma_{1}=\id_{\overline{K}}\}$ with $d_{1}=\#\,W_{1}$.  Let $W'_{i}$ denote the orbit of $a_{i}$. Then, since $q$ is injective, $\#\, W'_{i}\leq \#\,W_{1}=d_{1}$. Moreover,   for all $i=1,\ldots , n-1$,
$W'_{i}=\{\sigma_{j}(a_{i}) \mid j=1, \ldots , d_{1},\, \sigma_{1}=\id_{\overline{K}}\}$, but all $\sigma_{j}(a_{i})$, $j=1,\ldots , d_{1}$, may not be distinct.
\smallskip

Now, since $\sigma_{j}(a_{n})$, $j\!=\!1,\ldots , d_{1}$, are distinct elements in $\overline{K}$,  by \textit{Lagrange's Interpolation For\-mula}\,\footnote{\label{foot:8}\,Although named after Lagrange,\,J.\,L.\,(1736\,-\,1813) who published it in\,1795, the method was first discovered in\,1779 by Waring,\,E.\,(1734\,-\,1798).\,It is also an easy consequence of a formula of Euler,\,L.\,(1707\,-\,1783) published in\,1783. \, \
{\bf Lagrange's Interpolation Formula\,:} {\it Let $K$ be a field and let $x_{1},\ldots , x_{n}\in K$ be distinct elements. Then for arbitrary elements $y_{1},\ldots , y_{n}\in K$, there exists a polynomial $g\in K[X]$ of degree $\deg\,g<n$ such that $g(x_{i}) =y_{i}\,$ for every $i=1,\ldots , n$.} For a \so{proof}  consider the polynomial $\, \displaystyle g:=\sum_{i=1}^{n}\,\,\frac{y_{i}}{z_{i}}\,\prod_{j\neq i}\, \, (X-x_{j})$, where $\,\displaystyle  z_{i}:= \prod_{j\neq i}\, \,(x_{i}-x_{j})$.
}, for each $i\!=\!1,\ldots , n-1$, there exists a polynomial $g_{i}\!\in\!\overline{K}[X]$, $\deg\,g_{i}\! < \!d_{1}\!<\! \deg\,g_{n}$,  such that $g_{i}(\sigma_{j}(a_{n}))\!=\!\sigma_{j}(a_{i})$ for all $j\!=\!1,\ldots , d_{1}\!$.  Moreover, $g_{1}, \ldots , g_{n-1}\!\in\!K[X]$.
\smallskip

Finally we claim the equality $\mathfrak{A}':=\langle X_{1}-g_{1}(X_{n}), \ldots , X_{n-1}-g_{n-1}(X_{n}),\, g_{n}(X_{n})\rangle =\mathfrak{A}$. To prove this first  note that
the substitution homomorphism $K[X_{1},\ldots , X_{n-1}, X_{n}]\rightarrow K[X_{n}]$, $X_{i}\mapsto X_{i}-g_{i}(X_{n})$, $i=1,\ldots , n-1$ and $X_{n}\mapsto g_{n}(X_{n})$, induces a $K$-algebra isomorphism
 $K[X_{1}, \ldots , X_{n}]/\mathfrak{A}' \iso K[X_{n}]/\langle g\rangle$ and  $K[X_{1}, \ldots , X_{n}]/\mathfrak{A}' $  is reduced, since $g_{n}$ is separable over $K$. Therefore   $\mathfrak{A}'$ is a radical ideal. Further, from 5.2.a it follows that ${\rm V}_{\overline{K}}(\mathfrak{A}') ={\rm V}_{\overline{K}}(\mathfrak{A})$. Now, use Hilbert's Nullstellensatz (see\,\cite{AM}, \cite{Kunz1980} or  \cite[Theorem\,2.10\,,\,HNS\,2]{PatilStorch2010}) to conclude the equality $\mathfrak{A}'=\mathfrak{A}$.\qed
\end{proof}
\smallskip

\begin{remark}\label{rem:5.3}
The Shape Lemma\,\ref{lem:5.2} appeared first time in \cite{GianniMora1989} which  may be regarded as a natural generalization of the Primitive Element Theorem.  Further, it gives a very useful presentation of the radical ideal $\mathfrak{A}$ which allows to find the solution space ${\rm V}_{\overline{K}}(\mathfrak{A})$ immediately, namely\,: \\[1mm] \hspace*{20mm} $\,{\rm V}_{\overline{K}}(\mathfrak{A}) = \{(g_{1}(a), \ldots , g_{n-1}(a), a)\in \overline{K}^{n} \mid g_{n}(a) =0\,\}$. \\[1mm]
In other words the last coordinates are zeros of $g_{n}$ and for a fixed last coordinate $a_n$, all the other coordinates are determined by evaluation of polynomials $g_{n-1}, \ldots , g_{1}$ at $a_n$\,: \, $\,g_{n}(a_n) = 0$, $\,a_{n-1} = g_{n-1}(a_n),\, \ldots \, , a_1 = g_1(a_n)$. This simple shape of the solution space $\,{\rm V}_{\overline{K}}(\mathfrak{A})$ is quite convenient to work with.
The primary decomposition of $\mathfrak{A}$ is given by the prime factorization of the polynomial $g_{n}$.
Under the conditions on the polynomials $g_{1}, \ldots , g_{n-1}$, $g_{n}\in K[X]$ as in the proof of the Shape Lemma\,\ref{lem:5.2}, one can easily verify that $\,X_{1}-g_{1}(X_{n}), \ldots , X_{n-1}-g_{n-1}(X_{n}),\, g_{n}(X_{n})\,$
form a reduced (=\,minimal) Gr\"obner basis of the radical ideal $\,\mathfrak{A}\,$ relative to the lexicographic order $X_{1}>X_{2}>\cdots >X_{n}$. For a different proof of the Shape Lemma~\ref{lem:5.2} see\,\cite[Theorem\,3.7.25]{KR2000} and a detailed recipe for solving systems of polynomial equations efficiently using the Shape Lemma\,\ref{lem:5.2} is also given in \cite[Theorem\,3.7.26]{KR2000}. The Shape Lemma\,\ref{lem:5.2} also appeared in \cite[Ex.\,16, \S\,4, Ch.\,2]{CLO2005}.
\end{remark}
\vskip2pt

\begin{mypar}\label{mypar:5.4}{\sc Consequence and identifcation}\,
Let $K$ be a real closed field,  $\K:=\C_{K}=K[\,{\rm i}\,]$, ${\rm i}^{2}=-1$, the algebraic closure of $K$  {\rm (see\,\ref{mypar:3.1})} and let $\mathfrak{A}\subseteq K[X_{1},\ldots , X_{n}]$
a  radical ideal. Suppose  that  $A:=K[X_{1},\ldots , X_{n}]/\mathfrak{A}$ is a finite dimensional $K$-vector space.  \smallskip

Let $g_{1}, \ldots , g_{n-1}, g:=g_{n}(X)\in K[X]$ are the polynomials as in the statement of the  Shape Lemma\,{\rm \ref{lem:5.2}} and let $\varphi : A \iso K[X]/\langle g\rangle$ be the $K$-algebra isomorphism
as in the proof of   the  Shape Lemma\,{\rm \ref{lem:5.2}}. Then, since $g$ is square-free and $K$ is a real closed field (see Footnote\,\ref{foot:2}),   $g = (X-a_{1})\cdots (X-a_{r}) \pi_{1} \cdots \pi_{s}\,$,  $a_{i}\in K$, $i=1,\ldots r\,$  and $\pi_{j}=(X-z_{j})(X-\overline{z}_{j})\in K[X]$, $z_{j}\!\in\!\K\!\smallsetminus K$, $j\!=\!1,\ldots , s$, where  $r$, $s$ and $m\!=\!r\!+\!2s$ as in {\rm \ref{mypar:5.1}.a}, since $\varphi$ is a $K$-algebra iso\-morphism.
\smallskip

We use  the above $K$-algebra isomorphism $\varphi$ to identify  $\mathfrak{A}$ and $A$   with  $\langle g\rangle$ and $K[X]/\langle g\rangle$, respectively.  With this $\mathbcal{x}:=\{1, x, \ldots , x^{\,m-1}\}$ is a $K$-basis of  $A$, where $x$ is the image of $X$ in $A$ and
${\rm V}_{K}(\mathfrak{A}) = {\rm V}_{K}(g)=\{ a_{1},\ldots ,  a_{r}\}\subseteq
{\rm V}_{\K}(\mathfrak{A}) = {\rm V}_{\K}(g)=\{ a_{1},\ldots , a_{r}, z_{1}, \overline{z}_{1}, \ldots , z_{s}, \overline{z}_{s}\}$,  $r+2s =m$.
\smallskip

Further,  for $H\in K[X_{1}, \ldots , X_{n}]$, $H\neq 0$,  we put
$h(X):= H(g_{1}(X), \ldots , g_{n-1}(X), X)\in K[X]$.  Then using the above  identifications, we have
 $h(x)\in A$,  and the values $H({\bf a}_{i}) \in K$, $i=1,\ldots , r$, and $H({\bf a}_{r+j})$, $H(\overline{\bf a}_{r+j})\in \K$,
$j=1,\ldots , s$ are identified with the values  $h(a_{i}) \in K$,  $i=1,\ldots , r$, and  $h(z_{j})$,  $h(\overline{z}_{j})\in \K$, $j=1,\ldots , s$, respectively.

\end{mypar}
\smallskip

\begin{theorem}\label{thm:5.5}
With the notation   as in {\rm \ref{mypar:5.1}},   in {\rm \ref{mypar:5.4}},  let
$H\in K[X_{1}, \ldots , X_{n}]$, $H\neq 0$, $h$ be the image of $H$ in $A$ and let
$\,\Phi_{h}:A \times A \to K$, $(f,f')\mapsto \tr_{K}^{A}(hff')$, be the generalized trace form associated with $h\in A$. Then\,$:$
\smallskip

\begin{alist}

\item
The  Gram's matrix $\mscrG_{\Phi_{h}}(\mbcalx)$ of $\,\Phi_{h}$ with respect to the  $K$-basis $\mbcalx$ is a symmetric matrix in ${\rm M}_{m}(K)$. Moreover, $\,\mscrG_{\Phi_{h}}(\mbcalx)\!=\!\mscrV\,\mscrD\, {}^{\rm t}\mscrV$, where
$\mscrV\in \GL_{m}(\K)$ is the Vandermonde's matrix\,\footnote{\label{foot:9}\,{\bf Vandermonde's matrix}\, For elements $a_{1}, \ldots , a_{m}$ is a field $K$,  the matrix  $\mscrV(a_{1},\ldots , a_{m}):=(a_{i}^{\,j})_{\genfrac{}{}{0pt}{5}{\hspace*{-2.5mm}1\leq i\leq m}{0\leq j \leq m-1}}\in\M_{m}(K)$ is called the \so{Vanderminde's matrix of the elements} $a_{1}, \ldots , a_{m}$. The elements $a_{1}, \ldots , a_{m}$ are pairwise  distinct if and only if $\mscrV(a_{1},\ldots , a_{m})\in\GL_{m}(K)$.
} of the elements $a_{1},\ldots , a_{r}, z_{1}, \ldots , z_{s}, \overline{z}_{1}, \ldots , \overline{z}_{s}\in\K$\,  and\, $\mscrD\in \M_{m}(\K)$ is the diagonal matrix with diagonal entries $\, h( a_{1}), \ldots , h(a_{r}),  h(z_{1}), \ldots , h(z_{s}),  h(\overline{z}_{1}), \ldots , h(\overline{z}_{s})$.

\item
Let
$\,p_{H}:= \#\,\{{\bf a}\in {\rm V}_{K}(\mathfrak{A}) \mid H({\bf a})>0\,\}\,$ and   $\,q_{H}:= \#\,\{{\bf a}\in {\rm V}_{K}(\mathfrak{A}) \mid H({\bf a})<0\,\}\,$.
Then $\,\type\,\Phi_{h}\!=\! (p_{H}+s\,,\, q_{H}+s)$, where $s\!=\! \#\left({\rm V}_{\K}(\mathfrak{A})\!\smallsetminus\!{\rm V}_{K}(\mathfrak{A})\right)$ and $\rank \Phi_{h}\!=\,  \#\{{\bf a}\in {\rm V}_{\K}(\mathfrak{A}) \mid H({\bf a})\neq 0\}$. In~particular, $\,\sign\,\Phi_{h}= p_{H} - q_{H}\,$.
\end{alist}
\end{theorem}
\begin{proof}
Recall from \ref{mypar:5.1}  that\,: \\[0.5mm]
$\,{\rm V}_{K}(\mathfrak{A})=\{{\bf a}_{1},\ldots , {\bf a}_{r}\}\subseteq
{\rm V}_{\K}(\mathfrak{A})=\{{\bf a}_{1},\ldots , {\bf a}_{r}\,,\, {\bf a}_{r+1}\,,\, \overline{{\bf a}}_{r+1}\,,\, \ldots\,,\, {\bf a}_{r+s}\,,\, \overline{{\bf a}}_{r+s}\}$,
where $r:=\#\,{\rm V}_{K}(\mathfrak{A})$, $r+s =\#\,\Spm\, A$   and $m= r+2s=\Dim_{K} A= \Dim_{\K} A_{\K} =\#\,{\rm V}_{\K}(\mathfrak{A})$ and that
${\rm V}_{\K}(\mathfrak{A})$ is in general $X_{n}$-position, see 5.1.a and 5.1.b.
\smallskip

{\bf (a)}\, From the  indentifcations in \,\ref{mypar:5.4}, it follows that
for $1\leq k\,,\,\ell\leq m-1$, the $(k,\ell)$-entry in the Gram's matrix $\mscrG_{\Phi_{h}}(1,x, \ldots , x^{m-1})$ is\,: \\[1.5mm]
{\bf \ref{thm:5.5}.1}\,\hspace*{3.5mm}  $\displaystyle \tr_{K}^{A}(h(x)\,x^{k+\ell-2})
=\sum_{z\in {\rm V}_{\K}(g)}\!\! h(z)\, z^{\,k+\ell-2}\,$ \\[-0.75mm]
\hspace*{35mm} $\displaystyle
\,= \, \sum_{i=1}^{r}\! h(a_{i})\, a_{i}^{\,k+\ell-2}
+\sum_{j=1}^{s}\! \left(h(z_{j})\, z_{j}^{\,k+\ell-2}+h(\overline{z}_{j})\, \overline{z}_{j}^{\,k+\ell-2}\right)$.\\[0.75mm]
Now, by the Fundamental Theorem on Symmetric Polynomials (see\,\cite[Theorem\,54.13]{SchejaStorch1988},  the right hand side of \ref{thm:5.5}.1 is a polynomial in the coefficients of $h(X)$ and $g(X)$ (with coefficients in $\Z$) and hence belongs to $K$. Therefore
$\,\mscrG_{\Phi_{h}}(1,x, \ldots , x^{m-1})\,$ is  a symmetric matrix  in ${\rm M}_{m}(K)$.
Furthermore,  using the equation \ref{thm:5.5}.1,
the  equality
$\,\mscrG_{\Phi_{h}}(1,x, \ldots , x^{m-1}) =  \mscrV\,\mscrD {}^{\rm t}\mscrV$,
where $\mscrV$ and $\mscrD$ are as in the statement of {\bf (a)}, can be easily verified.  \\[1.5mm]
{\bf (b)}\, The assertion about the rank follows from the equality  $\rank \Phi_{h} = \rank \mscrG_{\Phi}(\mbcalx) = \rank\mscrD$, since  $\mscrV\in\GL_{m}(\K)$.
Further, note that the local decomposition $\,A\iso  K^{r}\times \K^{s}$ (see\,\ref{mypar:5.1}.c)   yields the orthogonal decomposition (see\,\ref{mypar:4.2} and \ref{decthm:2.13}) \\[1mm]
\hspace*{25mm} $\Phi_{h}\! =\! (\Phi_{h})_{1}^{K}\operp\cdots \operp (\Phi_{h})_{r}^{K}\operp (\Phi_{h})_{1}^{\K} \operp \cdots (\Phi_{h})_{s}^{\K}$, \\[1mm]
where $(\Phi_{h})_{i}^{K}\! = \! \Phi_{h}\vert K$,  is  the restrictions of $\Phi_{h}$  to the real component at $a_{i}\!\in\! K$ with  Gram's matrix  $\mscrG_{(\Phi_{h})_{i}^{K}} (1)\! =\!(h(a_{i}))\in{\rm M}_{1}(K)$,  $i\!=\!1,\ldots , r$) and $(\Phi_{h})^{\K}_{j}\!= \!\Phi_{h}\vert \K$,
is  the restrictions of $\Phi_{h}$ to the non-real component $\,K[X]/\langle \pi_{j}\rangle\iso \K\,$
at $\mscrM_{j}\! =\!\langle\pi_{j}\rangle \in\Spm\,A\!\smallsetminus\!\KSpec A$,  $j\!=\!1,\ldots , s$.  Furthermore, clearly, $\type\,(\Phi_{h})_{i}^{K}\!=\!\sign (\Phi_{h})_{i}^{K}\! =\! \sign h(a_{i})\!=\!\sign\, H({\bf a}_{i})$ for all  $i\!=\!1,\ldots , r$ and by Example\,\ref{ex:3.10} (since $\pi_{j}\!=\! (X\!-\!z_{j})(X\!-\!\overline{z}_{j})$, $z_{j}\in\K\!\smallsetminus\!K$), we have  $\,\type\,(\Phi_{h})_{j}^{\K}\!=\! (1,1)$ for all  $j\!=\!1,\ldots , s$.  Therefore, by Corollary\,\ref{cor:4.4}\,: \\[1mm]
\hspace*{20mm}$\type \,\Phi_{h}\!=\!\sum_{i=1}^{r} \, \type\,(\Phi_{h})_{i}^{K}\!+ \!\sum_{j=r+1}^{r+s}\, \type\, (\Phi_{h})_{j}^{\K}\!=\!(p_{H}+s\,,\, q_{H}+s)$\\[1mm]
and hence  $\,\sign\,\Phi_{h}\!=\!p_{H}\!-\!q_{H}$. \qed
\end{proof}
\smallskip

\begin{corollary}\label{cor:5.6}{\rm \so{(Hermite)}}
Let $K$ be an arbitrary real closed field and let  $\,g=b_{0}+b_{1}X+\cdots +b_{m-1}X^{m-1}+X^{m}\in K[X]$,  $\,A:=K[X]/\langle g\rangle$. Then the $\,\type\,\tr_{K}^{A} = (r+s\,,\, s)$, where $\,\tr_{K}^{A}:A\times A \to K$, $(f,f') \mapsto \tr_{K}(ff')$  is the trace form  on $A$,  $\,r=\#{\rm V}_{K}(g)$ is the number of zeros of $g$ in $K$ and $\,s\,$ is the half of the number of zeros of $g$ in the algebraic closure $\K$ of $K$  which are not in $K$. In~particular,
$\sign \, \tr_{K}^{A} = r = \#\,{\rm V}_{K}(g)$.
\end{corollary}

\begin{proof}
Using the notations as in the Theorem\,\ref{thm:5.5}, note that $\tr_{K}^{A}\!=\!\Phi_{1}$, where $1\!\in\!K[X]$ denote the constant polynomial. Therefore  $p_{1}\!=\!r\!=\!\#{\rm V}_{K}(g)$,  $q_{1}\!=\!0$  and
$\sign \, \tr_{K}^{A}\!=\!p_{1}-q_{1}\!=\! r\!=\!\#{\rm V}_{K}(g)$ by
by  \ref{thm:5.5}\,(c). Of course, the assertion also follows directly from Theorem\,\ref{thm:4.5}. \qed
\end{proof}
\smallskip

With the notation and assumptions as in \ref{mypar:5.1}, our main goal is to relate the cardinality $\#\,{\rm V}_{K}(\mathfrak{A})\,$ with the signatures of the generalized trace form on the finite $K$-algebra $A$.
\smallskip

\begin{mypar}\label{mypar:5.7}{\sc Notation}\, With the notation and assumptions as in \ref{mypar:5.1} and as \ref{mypar:5.4}. Further, let $H\in K[X_1,.....,X_n]$, $H\neq 0$ and ${\rm V}_{K}(H):= \{a\in K^{n}\mid H(a)=0 \}$ be the hypersurface (an $(n-1)$-dimensional affine algebraic set in $K^{n}$) defined by $H$. Then the complement of  ${\rm V}_{K}(H)$ in $K^{n}$ is the union of line-connected subsets (in the strong topology on $K^{n}$ (see\,Footnote\,\ref{foot:5}) on which $H$ takes either all positive values or all negative values, i.\,e. $K^{n}\!\smallsetminus\!{\rm V}_{K}(H)\!=\!H^+ \uplus H^-\!$,
where $H^+\!:=\!\{a\in  K^n \mid H(a)>0\}$, and $H^-\!:=\!\{a\in K^n\mid  H(a)<0\}$.

Further, since
${\rm V}_{K}(\mathfrak{A}) = \left({\rm V}_{K}(\mathfrak{A})\cap H^{+}\right) \biguplus \left({\rm V}_{K}(\mathfrak{A})\cap H^{-}\right) \biguplus \left({\rm V}_{K}(\langle\mathfrak{A}, H\rangle)\right)$, we have\,: \\[2mm]
{\bf \ref{mypar:5.7}.a} \hspace*{2mm} $\,\#\,{\rm V}_{K}(\mathfrak{A}) = \#\,\left({\rm V}_{K}(\mathfrak{A}) \cap  H^{+}\right)
+\, \#\,\left({\rm V}_{K}(\mathfrak{A})\cap H^{-}\right)
+ \, \#\,\left({\rm V}_{K}(\langle\mathfrak{A}, H\rangle)\right)\,$,\\[1mm]
and hence
to compute $\#\,{\rm V}_{K}(\mathfrak{A})\,$, we can use arbitrary polynomial $H\in K[X_{1},\ldots , X_{n}]$ and compute the cardinalities \ $\#\,{\rm V}_{K}(\mathfrak{A})\cap H^{+}$, \ $\#\,{\rm V}_{K}(\mathfrak{A})\cap H^{-}$ and \ $\#\,{\rm V}_{K}(\langle \mathfrak{A}, H\rangle)$.
\end{mypar}
\smallskip

More precisely, we have\,:
\smallskip

\begin{theorem}\label{thm:5.8}\,
With the notation and assumptions as in {\rm \ref{mypar:5.1}} and {\rm \ref{mypar:5.7}}. For $H\in K[X_1,.....,X_n]$, $H\neq 0$, let
$p_{H}:=\#\,{\rm V}_{K}(\mathfrak{A}) \cap H^{+}$ and  $q_{H}:=\#\,{\rm V}_{K}(\mathfrak{A}) \cap H^{-}$. Further, let $h$ denote the image of $H$ in $A=K[X_1,.....,X_n]/\mathfrak{A}$ and $\Phi_{h}:A \times A \to K$, $(f,g)\mapsto \tr_{K}^{A}(hfg)$  $($resp. $\Phi_{h^{2}}:A \times A \to K$,
$(f,g)\mapsto \tr_{K}^{A}(h^{2}fg))$  be the generalized trace forms defined
by $h\,($resp. by $h^{2})$ on $A$. Then\,:

\begin{alist}
\item {\rm (\so{Pederson-Roy-Szpirglas}\, \cite[Theorem\,2.1]{PRS1993})}\,  \\[1mm]
\hspace*{16mm}
  $\,\sign\,\Phi_{h}= p_{H}- q_{H}\,$ \ and \ $\,\rank\,\Phi_{h} =
\#\,\left({\rm V}_{\K}(\mathfrak{A})\!\smallsetminus\!{\rm V}_{\K}(H)\right)$.

\item   $\,\sign\,\Phi_{h^{2}}= p_{H}+q_{H}\,$ \ and \ $\,\rank\,\Phi_{h^{2}} =
\#\,\left({\rm V}_{\K}(\mathfrak{A})\!\smallsetminus\!{\rm V}_{\K}(H)\right)$.

\item
Let $\mathfrak{B}:=\langle \mathfrak{A}, H\rangle$ be the ideal $($in $K[X_{1},\ldots , X_{n}]\,)$ generated by $\mathfrak{A}$ and $H$. Then the $K$-algebra  $B:=K[X_{1},\ldots , X_{n}]/\mathfrak{B} $ is finite over $K$ and $\,\sign\,\tr^{B}_{K} = \#\,{\rm V}_{K}(\mathfrak{B})$.

\item
The three signatures $\,\sign\,\Phi_{h}$, $\,\sign\,\Phi_{h^{2}}$ and $\,\sign\,\tr_{K}^{B}$ uniquely determine the natural numbers $p_{H}$, $q_{H}\,$ \ and \  $\,\#\,{\rm V}_{K}(\mathfrak{B}) = {\rm V}_{K}(\mathfrak{A}) \cap {\rm V}_{K}(H)$. In~particular, they determine the cardinality $\,\#\,{\rm V}_{K}(\mathfrak{A}) = p_{h}+q_{h}+\#\,{\rm V}_{K}(\mathfrak{B})$.

\end{alist}
 \end{theorem}

\begin{proof}
{\bf (a)\,:}  Proved in Theorem\,\ref{thm:5.5}\,(b).
\smallskip

{\bf (b)\,:} Since $H^{2}(a)\!\!=\!H(a)\,H(a)\!>\!0$ for every $a\!\in\!H^{+}\!\cup H^{-}\!$ and ${\rm V}_{K}(H^{2})\!=\!{\rm V}_{K}(H)$, from  Theorem\,\ref{thm:5.5}\,(b) it follows that\,: \\[1mm]
\hspace*{15mm}   $\,\sign\,\Phi_{h^{2}}= p_{H}+q_{H}\,$ \  and
$\,\rank\,\Phi_{h^{2}}\!\!=\!
\#\,\left({\rm V}_{K}(\mathfrak{A})\!\smallsetminus\!{\rm V}_{K}(H)\right)$.
\smallskip

{\bf (c)\,:} Since the $K$-algebra $B$ is a homomorphic image of the $K$-algebra $A$, $B$ is also finite over $K$. The equality  $\,\sign\,\tr^{B}_{K} = \#\,{\rm V}_{K}(\mathfrak{B})$ is immediate from Theorem\,\ref{thm:5.5}\,(a) ($H=1$) or Theorem\,\ref{thm:4.5}.
\smallskip

{\bf (d)\,:} Immediate from the formula \ref{mypar:5.7}.a for $\#{\rm V}_{K}(\mathfrak{A})$ in \ref{mypar:5.7} and  (a) and (b). \qed
\smallskip
\end{proof}

\bibliographystyle{amsplain}

\end{document}